\documentclass[11pt]{amsart}
\usepackage{amsmath}%
\usepackage{amsfonts}%
\usepackage{amssymb}%
\usepackage{graphicx, color, enumerate}
\usepackage{amscd}
\usepackage{cancel}
\usepackage{mathrsfs}
\usepackage{latexsym}
\usepackage[
left=1.25in, right=1.25in]{geometry}
\usepackage{tikz}
 \usepackage{tikz-cd}
 \usetikzlibrary{matrix,arrows,decorations.pathmorphing}
\usetikzlibrary[patterns]
\newtheorem{theorem}{Theorem}[section]
\theoremstyle{plain}
\newtheorem*{theorem*}{Theorem}
\theoremstyle{plain}

\newtheorem{definition}{Definition}[section]
\newtheorem{lemma}{Lemma}[section]

\newtheorem{prop}{Proposition}[section]
\newtheorem{remark}{Remark}[section]

\numberwithin{equation}{section}
\DeclareRobustCommand{\rchi}{{\mathpalette\irchi\relax}}
\newcommand{\irchi}[2]{\raisebox{\depth}{$#1\chi$}}
\begin{document}
\title[]{Semilinear fractional elliptic problems with mixed Dirichlet-Neumann boundary conditions}

\author{J. Carmona}
\curraddr[J.~Carmona]{Departamento de Matem\'aticas,
Universidad de Almer{\'{\i}}a,\newline%
\indent Ctra. Sacramento s/n, La Ca\~nada de San Urbano, 04120 Almer\'{\i}a,   Spain}%
\email[J.~Carmona]{jcarmona@ual.es}%
\author{E. Colorado}
\curraddr[E.~Colorado]{Departamento de Matem\'aticas, Universidad Carlos III de Madrid\newline%
\indent Avenida de la Universidad 30, 28911 Legan\'es (Madrid), Spain}%
\email[E.~Colorado]{ecolorad@math.uc3m.es}%
\author{T.~Leonori}
\curraddr[T.~Leonori]{Dipartimento di Scienze di Base e Applicate per l'Ingegneria
Universit\`a di \newline%
\indent Roma \lq\lq Sapienza\rq\rq. Via Antonio Scarpa 10,
00161 Roma, Italy}%
\email[T.~Leonori]{tommaso.leonori@sbai.uniroma1.it}%
\author{A.~Ortega}
\curraddr[A.~Ortega]{Departamento de Matem\'aticas, Universidad Carlos III de Madrid\newline
\indent Avenida de la Universidad 30, 28911 Legan\'es (Madrid), Spain}%
\email[A.~Ortega]{alortega@math.uc3m.es
}%

\date{\today}
\thanks{E. Colorado and A. Ortega are partially supported
by the Ministry of Economy and Competitiveness of Spain and FEDER
under grant number MTM2016-80618-P, and  J. Carmona is partially
supported by Ministerio de Econom\'ia y Competitividad
(MINECO-FEDER), Spain under grant number MTM2015-68210-P and Junta
de Andaluc\'{\i}a under grant number FQM-194.}

\begin{abstract}
We study a nonlinear elliptic boundary value problem defined on a
smooth bounded domain involving the fractional Laplace operator, a
concave-convex powers term together with mixed Dirichlet-Neumann
boundary conditions.
\end{abstract}
\maketitle
\section{Introduction}
We study a nonlinear elliptic problem involving the fractional Laplace operator and a concave-convex power
term together with   mixed Dirichlet-Neumann boundary conditions. Namely,
\begin{equation}\label{problema_abajo}
\left\{\begin{array}{rcll}
(-\Delta)^su & \!\!\!= &\!\!\!\lambda u^q+u^r \quad& \mbox{in } \Omega,
\\
u& \!\!\!> & \!\!\!0&\mbox{in }  \Omega,
\\
u & \!\!\!= & \!\!\!0&\mbox{on } \Sigma_{\mathcal{D}},\\
\displaystyle\frac{\partial u}{\partial \nu} &\!\!\! = &\!\!\!0&\mbox{on } \Sigma_{\mathcal{N}},
\end{array}
\right.
\tag{$P_\lambda$}
\end{equation}
where $\Omega\subset \mathbb{R}^{N}$ is a bounded domain with
smooth boundary, $N>2s$, $(-\Delta)^s$, with $\frac 12<s<1$,
denotes the spectral fractional Laplace operator, $\lambda >0$ is
a real parameter and $0<q\le 1<r<\frac{N+2s}{N-2s}$. In order to
simplify the notation we denote the mixed boundary conditions as
\begin{equation}\label{concidiones}
B(u)=u\rchi_{\Sigma_{\mathcal{D}}}+\frac{\partial u}{\partial
\nu}\rchi_{\Sigma_{\mathcal{N}}},
\end{equation}
where $\chi_A$ stands for the characteristic function of a set $A$ and we assume that the boundary manifolds $\Sigma_{\mathcal{D}}$ and $\Sigma_{\mathcal{N}}$ are such that

\begin{equation*}
        (\mathfrak{B})\ \left\{
        {\renewcommand{\arraystretch}{1.2}\begin{tabular}{l}
       $\Sigma_{\mathcal{D}}$ and $\Sigma_{\mathcal{N}}$ are smooth $(N-1)$-dimensional submanifolds of $\partial\Omega$.\\
       $\Sigma_{\mathcal{D}}$ is a closed manifold of positive $(N-1)$-dimensional Lebesgue measure,\\
                 $\displaystyle\mkern+200mu|\Sigma_{\mathcal{D}}|=\alpha\in(0,|\partial\Omega|)$.\\
             $\Sigma_{\mathcal{D}}\cap\Sigma_{\mathcal{N}}=\emptyset\,,\ \Sigma_{\mathcal{D}}\cup\Sigma_{\mathcal{N}}=\partial\Omega\mbox{  and  }\Sigma_{\mathcal{D}}\cap\overline{\Sigma}_{\mathcal{N}}=\Gamma\,$ where $\Gamma$ is a smooth\\
            $(N-2)$-dimensional submanifold of $\partial\Omega$.
        \end{tabular}}
        \right.
\end{equation*}

\

Problems like \eqref{problema_abajo} have been studied in the last decades:
with the classical Laplace operator and Dirichlet boundary
condition, c.f. \cite{Lions} or \cite{ABC} for a deep study; with the Laplace operator and
mixed Dirichlet-Neumann boundary conditions, c.f.
\cite{ACP0,ACP,ColP}; with the $p$-Laplace operator, c.f.
\cite{BEP,AP,GaMP}; with fully nonlinear operators, c.f.
\cite{ChaCP}; and more recently with the fractional Laplace operator
and Dirichlet boundary conditions, c.f. \cite{BaCdPS,BCSS,BCdPS}.
Up to our knowledge, this is the first work where the concave-convex problem is analyzed with the spectral fractional Laplace operator associated with mixed
Dirichlet-Neumann boundary conditions.\newline
The main result proven in this work is the following.
\begin{theorem}\label{Thuno}
Assume that $\frac12<s<1$, $N>2s$ and $0< q \le 1<r<\frac{N+2s}{N-2s}$. Then
\begin{enumerate}
\item If $q=1$ there exists at least one solution to $(P_\lambda)$
for every $0<\lambda<\lambda_{1}^s$, where $\lambda_1^s$ denotes
the first eigenvalue of the spectral fractional Laplacian with the
boundary conditions \eqref{concidiones}, while there is no
solution for $\lambda \geq \lambda_1^s$.
Even more, there is a branch of solutions to $(P_\lambda)$
bifurcating from $(\lambda, u)=(\lambda_1^s,0)$, which cuts the
axis $\{\lambda =0\}$. \item If $0<q<1$ there exists
$0<\Lambda<\infty$ such that:
\begin{enumerate}
\item For $0<\lambda<\Lambda$ there is a minimal solution to \eqref{problema_abajo}. Moreover, the family of minimal solutions is increasing with respect
to $\lambda$.
\item For $\lambda=\Lambda$ there is at least one solution to \eqref{problema_abajo}.
\item For $\lambda>\Lambda$ there is no solution to \eqref{problema_abajo}.
\item Problem \eqref{problema_abajo} admits at least two solutions for every $0<\lambda<\Lambda$.
\end{enumerate}
\end{enumerate}
\end{theorem}

The following result deals with the sub-linear case $0<q<1$ and it provides a uniform $L^{\infty}(\Omega)$-bound for all the solutions to problems \eqref{problema_abajo} for any $0<\lambda\le\Lambda$.

\begin{theorem}\label{acotacion}
Assume that $\frac12<s<1$, $N>2s$, $0<q<1<r<\frac{N+2s}{N-2s}$. Then, there exists a constant $C=C(N,s,
\Omega, r,q)>0$ such that
\begin{equation*}
\sup\limits_{x\in\Omega} u_{\lambda}(x)\leq C,
\end{equation*}
for any solution $u_{\lambda}$ to
problems $(P_{\lambda})$  with $\lambda\in
[0,\Lambda]$, and $\Lambda$ defined in Theorem \ref{Thuno}.
\end{theorem}
We also obtain uniform $L^{\infty}$-estimates, in the case in which  we
move the boundary conditions. To be precise we consider a family of sets $\{\Sigma_{\mathcal{D}}(\alpha)\}$, with $\alpha\in (0, |\partial\Omega|]$ and $|\cdot|$ denoting the Lebesgue
measure in the appropriate dimension, such that:
\begin{itemize}
\item[(B$_1$)] $\Sigma_{\mathcal{D}}(\alpha)$ is connected or has a finite number of connected components.
\item[(B$_2$)] $\Sigma_{\mathcal{D}}(\alpha_1)\subset\Sigma_{\mathcal{D}}(\alpha_2)$ if $\alpha_1<\alpha_2$.
\item[(B$_3$)] $|\Sigma_{\mathcal{D}}(\alpha)|=\alpha$.
\end{itemize}
We call
$\Sigma_{\mathcal{N}}(\alpha)=\partial\Omega\backslash\Sigma_{\mathcal{D}}(\alpha)$
and we assume that
$\Sigma_{\mathcal{D}}(\alpha)\cap\overline{\Sigma}_{\mathcal{N}}(\alpha)=\Gamma(\alpha)$
is a $(N-2)$-dimensional smooth submanifold. For a family of this type we consider the corresponding family of mixed boundary value problems,
\begin{equation} \label{pal}
        \left\{
        \begin{tabular}{lcl}
        $(-\Delta)^su=\lambda u^q+u^r$ & &in $\Omega$, \\
        $u>0$& & in $\Omega$,\\
        $B_{\alpha}(u)=0$  & &on $\partial\Omega$, \\
        \end{tabular}
        \right.
        \tag{$P_{\alpha,\lambda}$}
\end{equation}
where $B_{\alpha}(u)$ is defined as $B(u)$ with
$\Sigma_{\mathcal{D}}$, $\Sigma_\mathcal{N}$ replaced by
$\Sigma_{\mathcal{D}}(\alpha)$, $\Sigma_\mathcal{N}(\alpha)$
satisfying the corresponding hypotheses $(\mathfrak{B}_{\alpha})$ and $(B_1)$-$(B_3)$. In this scenario we prove the following result.

\begin{theorem}\label{acotacion_unif}
Consider the family  $\{\Sigma_{\mathcal{D}}(\alpha)\}_{\alpha\in(0,|\partial \Omega|]}$  satisfying the hypotheses $(\mathfrak{B}_{\alpha})$ and $(B_1)$-$(B_3)$. For every $0<\varepsilon<|\partial\Omega|$, let us denote $I_{\varepsilon}=[\varepsilon,|\partial\Omega|]$ and let
\begin{equation*}
\mathcal{S}_{\varepsilon}=\{u:\Omega\rightarrow\mathbb{R}|\
\text{such that}\ u\ \text{is solution of}\ \eqref{pal},\mbox{ with } \alpha\in
I_{\varepsilon}\}.
\end{equation*}
Then, there exists a constant
$\mathcal{M}_{\varepsilon}>0$ such that
\begin{equation*}
\|u\|_{L^{\infty}(\Omega)}\leq\mathcal{M}_{\varepsilon},\quad \forall
u\in\mathcal{S}_{\varepsilon}.
\end{equation*}
\end{theorem}

In addition, we will also prove the following behavior for the minimal solutions as we move the boundary conditions.
\begin{theorem}\label{teo:movbc}
Consider the family  $\{\Sigma_{\mathcal{D}}(\alpha)\}_{\alpha\in(0,|\partial \Omega|]}$  satisfying the hypotheses $(\mathfrak{B}_{\alpha})$ and $(B_1)$-$(B_3)$. Then
\begin{enumerate}
\item the  minimal solutions $\{\underline{u}(\alpha)\}$ are uniformly bounded
for any $\alpha\in [0,|\partial\Omega|]$. Moreover,
\begin{equation*}
\|\underline{u}(\alpha)\|_{H^{s}(\Omega)},\ \|\underline{u}(\alpha)\|_{L^{\infty}(\Omega)}\rightarrow 0\ \text{as}\ \alpha \rightarrow 0;
\end{equation*}
\item the non minimal solutions (\textit{of mountain pass type}) are bounded and they converge to zero in $H^s(\Omega)$ as $\alpha\to 0$.
\end{enumerate}
\end{theorem}
The paper is organized as follows: In section 2, we introduce the appropriate functional framework for the spectral fractional Laplace operator.
In that section we also recall the extension technique due to Caffarelli and Silvestre, see \cite{CS}, that provides   an equivalent definition of the fractional Laplace operator via an auxiliary problem.
In section 3 we   study a half-space problem that will be useful in the proof of the main theorem;  we make use of the moving planes
method and we extend some results of \cite{DG} to the
fractional setting. Section 4 is devoted to   the concave-convex problem by means of certain limit problems, and we also prove Theorem \ref{acotacion}
 and Theorem \ref{acotacion_unif} which are based on the blow-up method of \cite{GS}. To accomplish this step we need some
 compactness properties that requires to know precise H\"older estimates for the solutions to mixed boundary problems. We use the results of \cite{CCLO} where the H\"older regularity of such solutions is proven. Section 5 is devoted to the proof of Theorem \ref{Thuno} and the behavior when we move the boundary conditions of some class of solutions.

\section{Functional setting and preliminaries}
As far as the fractional Laplace operator is concerned, we recall its definition given through the spectral decomposition. We closely follow the notation and framework of \cite{CCLO}.
Let $(\varphi_i,\lambda_i)$, $i\in \mathbb{N}$,  be the
eigenfunctions (normalized with respect to the $L^2(\Omega)$-norm)
and the eigenvalues of $(-\Delta)$ equipped with homogeneous mixed Dirichlet-Neumann boundary data, respectively.
Then the pairs  $(\varphi_i,\lambda_i^s)$,  $i\in \mathbb{N}$, turn out to be  the eigenfunctions and
eigenvalues of the fractional operator $(-\Delta)^s$.
Consequently, given two smooth functions $u_i(x)$, $i=1,2$,
 we have that  $\displaystyle u_i(x)=\sum_{j\geq1}\langle u_i,\varphi_j\rangle\varphi_j$,  and thus
\[
\langle
 (-\Delta)^s u_1, u_2
\rangle
=
\sum_{j\ge 1} \lambda_j^s\langle u_1,\varphi_j\rangle \langle u_2,\varphi_j\rangle,
\]
i.e., the action of the fractional operator on a function $u_1$ is given by
\begin{equation*}
(-\Delta)^su_1=\sum_{j\ge 1} \lambda_j^s\langle u_1,\varphi_j\rangle\varphi_j.
\end{equation*}
Hence the operator $(-\Delta)^s$ is well defined for functions that belong to  the fractional Sobolev Space   that vanish on $\Sigma_{\mathcal{D}}$.
Indeed for any smooth function we consider its spectral decomposition as
$$
u=\sum_{j\ge 1} a_j\varphi_j \qquad \mbox{with } \quad a_j  = \langle u,\varphi_j\rangle  \in \ell^2
$$
that allows us to define the following norm
$$
\|u\|_{H^s(\Omega)}^2=
\sum_{j\ge 1} a_j^2\lambda_j^s\,.
$$
Thus we define the Sobolev Space as
$$
H_{\Sigma_{\mathcal{D}}}^s(\Omega)= \overline{C^{\infty}_0 ( \Omega\cup \mathcal{N})}^{\| \cdot \|_{H^s(\Omega)}}.
$$
Observe that for any   $u\in H_{\Sigma_{\mathcal{D}}}^s(\Omega)$ then
\begin{equation*}
\|u\|_{H_{\Sigma_{\mathcal{D}}}^s(\Omega)}=\left\|(-\Delta)^{\frac{s}{2}}u\right\|_{L^2(\Omega)}.
\end{equation*}
As already stressed in \cite[Theorem 11.1]{LM}, if $0<s\le \frac{1}{2}$ then $H_0^s(\Omega)=H^s(\Omega)$ and, therefore, also $H_{\Sigma_{\mathcal{D}}}^s(\Omega)=H^s(\Omega)$, while for $\frac 12<s<1$, $H_0^s(\Omega)\subsetneq H^s(\Omega)$. Hence, the range $\frac 12<s<1$, for which we have $H_{\Sigma_{\mathcal{D}}}^s(\Omega)\subsetneq H^s(\Omega)$, provides   the correct functional space to study the mixed boundary problem \eqref{problema_abajo}.\newline
This definition of the fractional powers of the Laplace operator allows us to integrate by parts in the appropriate spaces, so that a natural definition of weak solution to problem \eqref{problema_abajo} is the following.
\begin{definition}
We say that a positive function $u\in H_{\Sigma_{\mathcal{D}}}^s(\Omega)$   is a solution to \eqref{problema_abajo} if
\begin{equation*}
\int_{\Omega}(-\Delta)^{s/2}u \,(-\Delta)^{s/2}\psi dx=\int_{\Omega}\left(\lambda u^q+u^r\right)\psi dx,\ \ \text{for all}\ \psi\in H_{\Sigma_{\mathcal{D}}}^s(\Omega).
\end{equation*}
\end{definition}
Following the previous definition, we can associate to problem \eqref{problema_abajo} the following energy functional,
\begin{equation}\label{funct:abajo}
I_\lambda(u)=\frac 12\int_{\Omega}|(-\Delta)^{s/2}u|^2dx-\frac{\lambda}{q+1}\int_{\Omega} |u|^{q+1}dx-\frac{1}{r+1}\int_{\Omega}|u|^{r+1} dx,\quad u\in H_{\Sigma_{\mathcal{D}}}^s(\Omega),
\end{equation}
whose critical points correspond to solutions of \eqref{problema_abajo}.

Working with the fractional operator $(-\Delta)^s$  it is well known that some difficulties arise when one tries to obtain explicit expressions
involving the action of the fractional Laplacian on, for example, products of functions. In order to overcome this difficulties,
we use the ideas of Caffarelli and Silvestre, see \cite{CS}, together with those of \cite{BCdPS, CT} to give an equivalent definition of the operator
$(-\Delta)^s$ by means of an auxiliary problem that we introduce next.\newline
Given a domain $\Omega$, we set the cylinder $\mathcal{C}_{\Omega}=\Omega\times(0,\infty)\subset\mathbb{R}_+^{N+1}$.
We denote with $(x,y)$ points that belong to $\mathcal{C}_{\Omega}$ and with
$\partial_L\mathcal{C}_{\Omega}=\partial\Omega\times[0,\infty)$ the lateral boundary of the  cylinder.

Let us also denote by $\Sigma_{\mathcal{D}}^*=\Sigma_{\mathcal{D}}\times[0,\infty)$ and $\Sigma_{\mathcal{N}}^*=\Sigma_{\mathcal{N}}\times[0,\infty)$ as well as $\Gamma^*=\Gamma\times[0,\infty)$.

It is clear that, by construction,
\begin{equation*}
\Sigma_{\mathcal{D}}^*\cap\Sigma_{\mathcal{N}}^*=\emptyset\,, \quad \Sigma_{\mathcal{D}}^*\cup\Sigma_{\mathcal{N}}^*=\partial_L\mathcal{C}_{\Omega} \quad \mbox{and} \quad \Sigma_{\mathcal{D}}^*\cap\overline{\Sigma_{\mathcal{N}}^*}=\Gamma^*\,.
\end{equation*}
 Given a function $u\in H_{\Sigma_{\mathcal{D}}}^s(\Omega)$ we define its $s$-extension, denoted by $U=E_{s}[u]$, as the solution to the problem
\begin{equation*}
        \left\{
        \begin{array}{rlcl}
        \displaystyle   -\text{div}(y^{1-2s}\nabla U)&\!\!\!\!=0  & & \mbox{ in } \mathcal{C}_{\Omega} , \\
        \displaystyle B(U)&\!\!\!\!=0   & & \mbox{ on } \partial_L\mathcal{C}_{\Omega} , \\
         \displaystyle U(x,0)&\!\!\!\!=u(x)  & &  \mbox{ on } \Omega\times\{y=0\} ,
        \end{array}
        \right.
\end{equation*}
where
\begin{equation*}
B(U)=U\rchi_{\Sigma_{\mathcal{D}}^*}+\frac{\partial U}{\partial \nu}\rchi_{\Sigma_{\mathcal{N}}^* },
\end{equation*}
being $\nu$, with an abuse of notation\footnote{Let $\nu$ be the outwards normal vector to $\partial\Omega$ and $\nu_{(x,y)}$
the outwards normal vector to $\mathcal{C}_{\Omega}$ then, by construction, $\nu_{(x,y)}=(\nu,0)$, $y>0$.}, the exterior normal to $\partial_L\mathcal{C}_{\Omega}$. Following the well known result by Caffarelli and Silvestre (see \cite{CS}), $U$ is related to the fractional Laplacian of the original function through the formula
\begin{equation*}
\frac{\partial U}{\partial \nu^s}:= -\kappa_s \lim_{y\to 0^+} y^{1-2s}\frac{\partial U}{\partial y}=(-\Delta)^su(x),
\end{equation*}
where $\kappa_s$ is a  suitable positive constant (see \cite{BCdPS} for its exact value). The extension function belongs to the space
\begin{equation*}
H_{\Sigma_{\mathcal{D}}^*}^1(\mathcal{C}_{\Omega},y^{1-2s}dxdy) : =\overline{\mathcal{C}_{0}^{\infty}((\Omega\cup\Sigma_{\mathcal{N}})\times[0,\infty))}^{\|\cdot\|_{H_{\Sigma_{\mathcal{D}}^*}^1(\mathcal{C}_{\Omega},y^{1-2s}dxdy)}},
\end{equation*}
that is a Hilbert space equipped with the norm induced by the scalar product
\begin{equation*}
\langle U, V \rangle_{H_{\Sigma_{\mathcal{D}}^*}^1(\mathcal{C}_{\Omega},y^{1-2s}dxdy)}=\kappa_s \int_{\mathcal{C}_{\Omega}}y^{1-2s} \langle\nabla U,\nabla V\rangle dxdy.
\end{equation*}
Moreover, the following inclusions are satisfied, for $\frac 12<s<1$,
\begin{equation} \label{embedd}
H_0^1(\mathcal{C}_{\Omega},y^{1-2s}dxdy)\subset H_{\Sigma_{D}^*}^1(\mathcal{C}_{\Omega},y^{1-2s}dxdy)\subsetneq H^1(\mathcal{C}_{\Omega},y^{1-2s}dxdy),
\end{equation}
with $H_0^1(\mathcal{C}_{\Omega},y^{1-2s}dxdy)$ the space of functions that belong to $H^1(\mathcal{C}_{\Omega},y^{1-2s}dxdy)$ and vanish on the lateral boundary of $\mathcal{C}_{\Omega}$.

Consequently  we can reformulate   problem \eqref{problema_abajo} in terms of the extension problem as follows:
\begin{equation}\label{extension_problem}
        \left\{
        \begin{array}{rlcl}
        \displaystyle   -\text{div}(y^{1-2s}\nabla U)&\!\!\!\!=0  & & \mbox{ in } \mathcal{C}_{\Omega} , \\
        \displaystyle B(U)&\!\!\!\!=0   & & \mbox{ on } \partial_L\mathcal{C}_{\Omega} , \\
        U &\!\!\!> 0  & &  \mbox{ on } \Omega\times\{y=0\}
        \\
         \displaystyle \frac{\partial U}{\partial \nu^s}&\!\!\!\!=\lambda\, U^q+U^r  & &  \mbox{ on } \Omega\times\{y=0\} .
        \end{array}
        \right.
        \tag{$P_{\lambda}^*$}
\end{equation}

 Hence we give a definition of energy solution of \eqref{extension_problem} in the following way.

\begin{definition}
An {\rm energy solution} to problem \eqref{extension_problem} is a function $U\in H_{\Sigma_{\mathcal{D}}^*}^1(\mathcal{C}_{\Omega},y^{1-2s}dxdy)$, with $U>0$ on $\Omega\times\{y=0\}$, such that
\begin{equation*}
\kappa_s\int_{\mathcal{C}_{\Omega}} y^{1-2s} \langle\nabla U,\nabla\varphi\rangle  \  dxdy=\int_{\Omega} \left(\lambda\, U^q(x,0)+U^r(x,0)\right)\varphi(x,0)dx,
\end{equation*}
for all $\varphi\in H_{\Sigma_{\mathcal{D}}^*}^1(\mathcal{C}_{\Omega},y^{1-2s}dxdy)$.
\end{definition}
For  any weak or energy solution $U\in H_{\Sigma_{\mathcal{D}}^*}^1(\mathcal{C}_{\Omega},y^{1-2s}dxdy)$ to problem \eqref{extension_problem} we can associate the function $u (x) =Tr[U(x,y) ]=U(x,0)$, that  belongs to $H_{\Sigma_{\mathcal{D}}}^s(\Omega)$, and solves problem \eqref{problema_abajo}.
Moreover,  the viceversa is true: given a solution $u \in H_{\Sigma_{\mathcal{D}}}^s(\Omega)$ we can define its $s$-extension $U(x,y)$ as
a solution of \eqref{extension_problem} with $U\in H_{\Sigma_{\mathcal{D}}^*}^1(\mathcal{C}_{\Omega},y^{1-2s}dxdy)$.
Thus, both formulations are equivalent and the {\it Extension operator}
$$
E_s: H_{\Sigma_{\mathcal{D}}}^s(\Omega) \to H_{\Sigma_{\mathcal{D}}^*}^1(\mathcal{C}_{\Omega},y^{1-2s}dxdy),
$$
allows us to switch from \eqref{problema_abajo} to \eqref{extension_problem}.

According  with \cite{CS,BCdPS}, due to the choice of the constant $\kappa_s$, the extension operator $E_s$ is an isometry, i.e.,
\begin{equation*}
\|E_s[\varphi] (x,y) \|_{H_{\Sigma_{\mathcal{D}}^*}^1(\mathcal{C}_{\Omega},y^{1-2s}dxdy)}=\|\varphi (x) \|_{H_{\Sigma_{\mathcal{D}}}^s(\Omega)},\ \forall\ \varphi\in H_{\Sigma_{\mathcal{D}}}^s(\Omega).
\end{equation*}
It is also proved in \cite{BCdPS} that, given $z\in H_0^1(\mathcal{C}_{\Omega},y^{1-2s}dxdy)$, there exists $C_0=C_0(N,s,r,|\Omega|)$ such that the \textit{trace inequality},
\begin{equation*}
\int_{\mathcal{C}_{\Omega}}y^{1-2s}|\nabla z(x,y)|^2dxdy\geq C_0
\left(\int_{\Omega}|z(x,0)|^rdx\right)^{\frac{2}{r}},
\end{equation*}
holds provided $1\leq r\leq 2^*_s,\ N>2s$, where $2^*_s= \frac{2N}{N-2s}$ is the critical fractional Sobolev exponent.
Such inequality turns out to be very useful and it is in fact equivalent to the fractional Sobolev inequality,
\begin{equation*}
\int_{\Omega}|(-\Delta)^{\frac{s}2}v|^2dx\geq C_0\left(\int_{\Omega}|v|^rdx\right)^{\frac{2}{r}},\qquad \forall v\in H_{0}^s(\Omega),\ 1\leq r\leq 2^*_s ,\ N>2s.
\end{equation*}
When mixed boundary conditions are considered, the situation is quite similar since the Dirichlet condition is imposed on a set
$\Sigma_{\mathcal{D}} \subset \partial \Omega$ such that $|\Sigma_{\mathcal{D}}|=\alpha>0$. Hence, thanks to \eqref{embedd},
there exists a positive constant $S(\Sigma_{\mathcal{D}})=S(N,s,\Sigma_{\mathcal{D}},\Omega)$ such that
\begin{equation*}
0<S(\Sigma_{\mathcal{D}}):=\inf_{\substack{u\in H_{\Sigma_{\mathcal{D}}}^s(\Omega)\\ u\not\equiv 0}}\frac{\|u\|_{H_{\Sigma_{\mathcal{D}}}^s(\Omega)}^2}{\|u\|_{L^{2_s^*}(\Omega)}^2}
\le\inf_{\substack{u\in H_{0}^s(\Omega)\\ u\not\equiv 0}}\frac{\|u\|_{H_{0}^s(\Omega)}^2}{\|u\|_{L^{2_s^*}(\Omega)}^2}.
\end{equation*}
\begin{remark}\label{rem:sobconst}
Actually, $S(\Sigma_{\mathcal{D}})\leq 2^{-\frac{2s}{N}}C_0(N,s)$, see \cite{CO}. Moreover, taking in mind the spectral definition of the fractional
operator and making use of the H\"older inequality, it follows that $S(\Sigma_{\mathcal{D}})\leq |\Omega|^{\frac{2s}{N}}\lambda_1^s(\alpha)$, with
$\lambda_1(\alpha)$ the first eigenvalue of the Laplace operator with mixed boundary conditions on the sets $\Sigma_{\mathcal{D}}=
\Sigma_{\mathcal{D}}(\alpha)$ and $\Sigma_{\mathcal{N}}=\Sigma_{\mathcal{N}}(\alpha)$. Under geometrical assumptions $(B_1)$-$(B_3)$ one has that, by
\cite[Lemma 4.3]{ColP}, $\lambda_1(\alpha)\to 0$ as $\alpha\searrow 0$
which shows that $S(\Sigma_{\mathcal{D}})\to0$ as $\alpha\searrow 0$.
\end{remark}
Then, in analogy with the Dirichlet boundary data case, the following mixed trace inequality holds (see \cite{CCLO}).
\begin{lemma}\label{lem:traceineq}
There exists a constant $C=C(N,s,r,\Sigma_{\mathcal{D}},\Omega)>0$ such that,
\begin{equation}\label{eq:traceineq}
\int_{\mathcal{C}_{\Omega}} y^{1-2s} |\nabla \varphi|^2 dxdy\geq C
\left(\int_\Omega
|\varphi(x,0)|^{r} dx\right)^{\frac{2}{r}},
\end{equation}
for all $\varphi \in H_{\Sigma_{\mathcal{D}}^*}^1(\mathcal{C}_{\Omega},y^{1-2s}dxdy)$ and $1\leq r\leq 2^*_s,\ N>2s$, where $2^*_s= \frac{2N}{N-2s}$.
\end{lemma}
As a consequence,
\begin{equation*}
\int_{\Omega}|(-\Delta)^{\frac{s}2}v|^2dx\geq \kappa_s C\left(\int_{\Omega}|v|^rdx\right)^{\frac{2}{r}},\qquad \forall v\in
H_{\Sigma_{\mathcal{D}}}^s(\Omega),\ 1\leq r\leq 2^*_s ,\ N>2s.
\end{equation*}
Note that in case $r=2^*_s$, then $\kappa_s C=S(\Sigma_{\mathcal{D}}).$
\section{Moving planes and monotonicity}
In this section we establish a monotonicity result for bounded
solutions to $(-\Delta)^s u =u^r$ in $\mathbb{R}_{+}^{N}\equiv
\mathbb{R}^{N-1} \times \mathbb{R}_{+}$ satisfying the boundary
conditions:
\begin{itemize}
 \item  $u=0$ on $\Sigma_{\mathcal{D}}(\tau)=\{(x_1,\cdots, x_N)\in \mathbb{R}^N: x_N=0,
x_1\leq \tau\},$
for some $\tau\in \mathbb{R}$.
\item $\frac{\partial u}{\partial x_N}=0$ on $\Sigma_{\mathcal{N}}(\tau)\Sigma_{\mathcal{N}}(\tau)=\{(x_1,\cdots, x_N)\in \mathbb{R}^N: x_N=0,
x_1 > \tau\}$,
for some $\tau\in \mathbb{R}$.
\end{itemize}
The principal result proven in this section is the following.
\begin{theorem}\label{thmonotonia} Assume that $1<r<\frac{N+2s}{N-2s}$, $N>2s$, and $\tau\in\mathbb{R}$. Let $u\in H_{loc}^s(\mathbb{R}_{+}^{N})\cap\mathcal{C}^0(\overline{\mathbb{R}_{+}^{N}})$ be a weak solution to
\begin{equation} \label{prob:mov:abajo}
        \left\{
        \begin{tabular}{rlr}
        $\displaystyle (-\Delta)^su=u^r$, &$u>0$, &in $\mathbb{R}_{+}^{N}$, \\
        $ u=0\ \ $ & &on $\ \Sigma_{\mathcal{D}}(\tau)$,\\
        $\frac{\partial u}{\partial x_N}=0\ \ $ & &\ on $\ \Sigma_{\mathcal{N}}(\tau)$.
        \end{tabular}
        \right.
\end{equation}
Then, $u$ is
nondecreasing with respect to the $x_1$-direction.
\end{theorem}

\begin{remark}
We make the proof assuming $\tau=0$. For $\tau\neq0$ the proof is analogous through a translation with respect to  the variable $x_1$.
\end{remark}

The proof of Theorem \ref{thmonotonia} is based on the method of moving planes introduced by Alexandrov and first exploited
in the context of Partial Differential Equations by J. Serrin \cite{Serr}, see also \cite{GLN} for more details.

\smallskip

 Let us introduce some notation in order to apply the moving planes method.
 We denote by   $\mathbb{R}_{++}^{N+1}\equiv\mathbb{R}_{+}^{N}\times \mathbb{R}_{+}$, i.e., the set of points $X=(x,y)$ with
 $x=(x_1,\ldots,x_N)$ and $x_N, y>0$. For a fixed $\rho\in \mathbb{R}$,  we define the sets
$$\Upsilon_{\rho}=\{x\in\mathbb{R}_{+}^{N}:\ x_1<\rho\},\quad \Upsilon_{\rho}^*=\Upsilon_{\rho}\times \mathbb{R}_+,$$
$$T_{\rho}=\{X\in\overline{\mathbb{R}_{++}^{N+1}}:\ x_1=\rho\}.$$
For any $X\in\mathbb{R}_{++}^{N+1}$ the reflection with respect to the hyperplane  $T_{\rho}$ is denoted by
$$X^{\rho}=(x^{\rho},y)=X+2(\rho-x_1)e_1=(2\rho-x_1, x_2,\ldots,x_N,y).$$
Let us define the point $O_{\rho}=(2\rho,0,\ldots,0,0)\in \mathbb{R}^{N+1}$, whose reflection is the origin, and $o_{\rho}=(2\rho,0,\ldots, 0)\in \mathbb{R}^N$.
We also recall that the Kelvin transform of a nontrivial point $x\in \mathbb{R}^N$ is given by $\mathcal{K}(x)=\frac{x}{|x|^2}$. It is easy to see that $\mathcal{K}(\mathbb{R}_+^N)=\mathbb{R}_+^N$ and $\mathcal{K}\left(\Upsilon_{\rho}^*\right)=(\mathbb{R}_{++}^{N+1})\cap B_{\frac{1}{-4\rho}}(O_{\frac{1}{4\rho}})$ for any $\rho<0$.
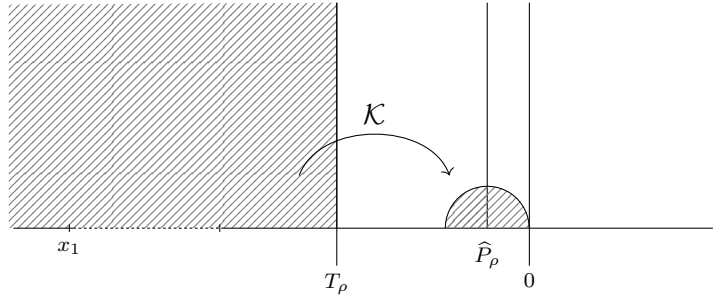
\begin{figure}[!ht]
\begin{center}
\begin{tikzpicture}[scale=1]
[line cap=round,line join=round,>=triangle 45,x=1.0cm,y=1.0cm]
\draw[fill=gray, draw=none,pattern=north east lines, pattern color= gray] (-4.360242255065743,2.9744989159594843) -- (0.,3.) -- (0.,0.) -- (-4.360242255065743,0.) -- cycle;
\path[->, bend left=70]  (-0.5,0.7) edge node[auto] {$\mathcal{K}$} (1.5,0.7);
\begin{scriptsize}
\draw (0.,3.)-- (0.,-0.5) node[below] {$T_{\rho}$};
\draw (-1.5,0.)-- (5.,0.);
\draw (2.56,3.)-- (2.56,-0.5) node[below] {$0$};
\draw (2.,3.)-- (2.,0.) node[below=2] {$\widehat{P}_{\rho}$};
\draw [dash pattern=on 1pt off 1pt] (-1.5,0.)-- (-3.5,0.); 
\draw (-3.5,0.)-- (-4.3,0.);
\draw  (0.,3.)-- (0.,0.);
\draw (-1.55556,-0.05) rectangle (-1.555552,0.05) node[below=3.5] {$$};
\draw (-3.55556,-0.05) rectangle (-3.555552,0.05) node[below=3.5] {$x_{1}$};
\draw [fill=gray,pattern=north east lines, pattern color= gray,shift={(2.,0.)}] plot[domain=0:3.141592653589793,variable=\t]({1.*0.5578767958337079*cos(\t r)+0.*0.5578767958337079*sin(\t r)},{0.*0.5578767958337079*cos(\t r)+1.*0.5578767958337079*sin(\t r)});
\fill [gray,pattern=north east lines, pattern color= gray] (2.,0.) -- (1.4338273018014687,0.) -- (2.,0.5578767958337079) -- cycle;
\end{scriptsize}
\end{tikzpicture}
  \caption{The Kelvin Transform acting on the set $\Upsilon_{\rho}^*$, with $\rho <0$.}
\end{center}
\end{figure}
Next, we follow an approach similar to the one in \cite{BCdPS} based on the fractional Kelvin transform, $\mathcal{K}_{s}(u)$, which acts on functions
defined in a subset of $\mathbb{R}^N,$ in the following way:
$$\mathcal{K}_{s}(u)=\frac{1}{|x|^{N-2s}}u\left(\mathcal{K}(x)\right)=\frac{1}{|x|^{N-2s}}u\left(\frac{x}{|x|^2}\right).$$
As it is proven in \cite{BCdPS}, if $(-\Delta)^su=f(u)$, then the action of the fractional laplacian acting on the fractional Kelvin transform of $u$ is given by
\begin{equation*}
(-\Delta)^s\mathcal{K}_{s}(u)=\frac{1}{|x|^{N+2s}}f\left(u(\mathcal{K}(x))\right).
\end{equation*}
Let $u(x)$ be a solution to problem \eqref{prob:mov:abajo} and define $f(t)=t^r$ and $\displaystyle g(t)=\frac{f(t)}{t^{\frac{N+2s}{N-2s}}}$. Then, the Kelvin transform $v=\mathcal{K}_{s}(u)$ satisfies the
following mixed BVP,
\begin{equation*}
        \left\{
        \begin{tabular}{rlr}
        $\displaystyle(-\Delta)^sv=g(|x|^{N-2s}v)v^{\frac{N+2s}{N-2s}}$,& $v>0$, & in $\mathbb{R}_{+}^{N}$, \\
        $v=0\mkern+132.5mu$& &on $\ \Sigma_{\mathcal{D}}(0)$,\\
        $\displaystyle\frac{\partial v}{\partial x_N}=0\mkern+132.5mu$&  &on $\ \Sigma_{\mathcal{N}}(0)$,
        \end{tabular}
        \right.
\end{equation*}
since on $x_N=0$, we have
$$
\frac{\partial v}{\partial x_N}(x)=(2s-N)\frac{x_N}{|x|^{N+2(1-s)}}u\left(\mathcal{K}(x)\right)+\frac{1}{|x|^{N-2s}}\frac{\partial u}{\partial x_N}\left(\mathcal{K}(x)\right)=0.
$$
Moreover, $v$ is a continuous and positive function in
$\mathbb{R}^N\backslash\{0\}$, with a possible singularity at the
origin and decays at infinity as $\frac{1}{|x|^{N-2s}}u(0)$, thus
$v\in L^{2_s^*}\cap L^{\infty}(\mathbb{R}_+^N\backslash B_r(0))$ for any $r>0$.
Finally, we consider $V=E_s[v]$ the extension function of the Kelvin transform
$v=\mathcal{K}_{s}(u)$ and the corresponding extension problem,
\begin{equation}\label{extension_kelvin}
        \left\{
       \begin{array}{rlcl}
        \displaystyle   -\text{div}(y^{1-2s}\nabla V)&\!\!\!\!=0  & & \mbox{ in }\mathbb{R}_{++}^{N+1}\subset\mathbb{R}_{+}^{N+1} , \\
        \displaystyle B(V)&\!\!\!\!=0   & & \mbox{ on } \left(\Sigma_{\mathcal{D}}(0)\cup\Sigma_{\mathcal{N}}(0)\right)\times\mathbb{R}_{+} , \\
         \displaystyle \frac{\partial U}{\partial \nu^s}&\!\!\!\!=g(|x|^{N-2s}v)v^{\frac{N+2s}{N-2s}}& &  \mbox{ on } \Omega\times\{y=0\} .
        \end{array}
        \right.
\end{equation}
Observe that, since $v\in L^{2_s^*}(\mathbb{R}_+^N\backslash B_r(0))$ for any $r>0$ and the extension operator $E_s$ is an isometry, by \cite{FKS}, the extension function $V\in L^{\overline{2}^*}(\Upsilon_{\rho}^*,y^{1-2s}dX)$ for any $\rho<0$, where $\overline{2}^*=\frac{2(N+1)}{N-1}$ denotes to the Sobolev conjugate exponent in dimension $N+1$.\newline
The following lemma, which extends to our fractional framework \cite[Lemma 2.1]{DG}, provides us with a key-point inequality in order to obtain monotonicity in the $x_1$-direction for the function $V$ defined in \eqref{extension_kelvin}.

\smallskip

Here we use the notation $V_\rho(X)=V(X^\rho)$ and $v_\rho(x)=v(x^\rho)$ for the reflected functions that are singular at the point $O_{\rho}$ and $o_\rho$ respectively. Moreover we denote by  $\mathcal{A}_{\rho}=\{x\in \Upsilon_{\rho}\backslash O_{\rho}: v\geq v_{\rho}\}$.

\begin{lemma}\label{integralDG}
Assume that $u\in H_{loc}^s(\mathbb{R}_{+}^{N})\cap\mathcal{C}^0(\overline{\mathbb{R}_{+}^{N}})$ is a weak solution of \eqref{prob:mov:abajo} and let $v=\mathcal{K}_{s}(u)$.
Then, for any $\rho<0$, 
$(v-v_{\rho})^+\!\!\in\! H_{\Sigma_{\mathcal{D}}}^s(\Upsilon_{\rho})\cap L^{\infty}(\Upsilon_{\rho})$. Moreover, there exists $C_{\rho}>0$, increasing with respect to  $\rho$, such that
\begin{equation}\label{ineq:lemmaDG}
\int_{\Upsilon_{\rho}^*}y^{1-2s}|\nabla(V-V_{\rho})^+|^2dxdy\leq
C_{\rho}\left(\int_{\mathcal{A}_{\rho}}\frac{1}{|x|^{2N}}dx\right)^{\frac{2s}{N}}\int_{\Upsilon_{\rho}^*}y^{1-2s}|\nabla(V-V_{\rho})^+|^2dxdy.
\end{equation}
\end{lemma}
\begin{proof}
Since for a given $\rho<0$ there exists $r>0$ such that $\Upsilon_{\rho}\subset \mathbb{R}_+^N\backslash B_r(0)$, the functions $v$ and $(v-v_{\rho})^+\leq v$ belong to $L^{2_s^*} (\Upsilon_{\rho}) \cap L^{\infty}(\Upsilon_{\rho})$ and the function $\frac{1}{|x|^{2N}}$ is integrable in $\Upsilon_{\rho}$. The assertion $(v-v_{\rho})^+\!\!\in\! H_{\Sigma_{\mathcal{D}}}^s(\Upsilon_{\rho})$ follows from \eqref{ineq:lemmaDG} taking in mind that the extension operator $E_s$ is an isometry.
To prove inequality \eqref{ineq:lemmaDG} we test conveniently the equations
\begin{equation*}
(-\Delta)^sv=g(|x|^{N-2s}v)v^{\frac{N+2s}{N-2s}},\quad (-\Delta)^sv_{\rho}=g(|x^{\rho}|^{N-2s}v_{\rho})v_{\rho}^{\frac{N+2s}{N-2s}},
\end{equation*}
in the set $\Upsilon_{\rho}\backslash O_{\rho}$.
At this point, we make full use of the extension technique, so that we consider the extension functions $V=E_s[v]$ and $V_\rho=E_s[v_{\rho}]=V(X^{\rho})$ and we set the nonnegative function $\varphi=\varphi_{\varepsilon}=\eta_{\varepsilon}^{2}(V-V_{\rho})^{+}$ as a test function in the corresponding extended problem for a convenient function $\eta_{\varepsilon}$. More precisely, for $\varepsilon>0$ small enough we take $\eta_\varepsilon \in \mathcal{C}_0^1(\mathbb{R}^{N+1})$ with $0\leq\eta_\varepsilon\leq 1$ and such that:

 $$
\left\{ \begin{array}{ll}
 \eta_\varepsilon\equiv 1 \qquad & \mbox{  for } \displaystyle2\varepsilon\leq|X-O_{\rho}|\leq\frac{1}{\varepsilon}\\[1.5 ex]
 \eta_\varepsilon\equiv 0 &\mbox{ for } |X-O_{\rho}|\leq\varepsilon \quad \mbox{ or  } \quad \displaystyle\frac{2}{\varepsilon}\leq|X-O_{\rho}|, \\[1.5 ex]
 \displaystyle|\nabla\eta_\varepsilon|\leq\frac{c}{\varepsilon} &\mbox{ for } \varepsilon<|X-O_{\rho}|<2\varepsilon \\[1.5 ex]
 \displaystyle|\nabla\eta_\varepsilon|\leq c \varepsilon &\mbox{ for } \displaystyle\frac{1}{\varepsilon}<|X-O_{\rho}|<\frac{2}{\varepsilon}.
 \end{array}
\right. $$

Observe that in the set $\Upsilon_{\rho}^*$ the function $(V-V_{\rho})^+$ vanishes where the Dirichlet condition holds for $V$ but also where the Dirichlet condition holds for the reflected function and, therefore, it is allowed to take $\varphi=\eta_{\varepsilon}^{2}(V-V_{\rho})^{+}$ as a test function in the corresponding extended problem.
\begin{figure}[!ht]
\begin{center}
\begin{tikzpicture}[scale=0.63][line cap=round,line join=round,>=triangle 45,x=1.0cm,y=1.0cm]
\path[->, bend left=70]  (5,0.7) edge node[auto] {$\mathcal{K}$} (6.5,0.7); 
\draw [line width=0.7pt,dotted] (-4.,0.)-- (0.,0.);
\draw [line width=0.7pt] (0.,0.)-- (4.,0.);
\draw [line width=0.7pt,dotted] (7.,0.)-- (11.,0.);
\draw [line width=0.7pt,dotted, color=red] (8.05,0.)-- (9.5,0.);
\draw [line width=0.7pt] (11.,0.)-- (15.,0.);
\draw [line width=0.7pt,dotted] (8.,-1.)-- (15.,-1.);
\draw [line width=0.7pt,dotted, color=red] (8.,-1.)-- (9.5,-1.);
\draw [line width=0.7pt] (7.,-1.)-- (8.,-1.);
\draw [line width=0.7pt] (0.,2.)-- (0.,0.);
\draw [line width=0.7pt] (11.,0.)-- (11.,2.);
\draw [line width=0.7pt,dash pattern=on 4pt off 4pt] (8.,2.)-- (8.,-1.);
\draw [line width=0.7pt,dash pattern=on 4pt off 4pt] (9.5,-1.)-- (9.5,2.);
\begin{scriptsize}
\draw [fill=black,shift={(-4.,0.)},rotate=90] (0,0) ++(0 pt,3.75pt) -- ++(3.2475952641916446pt,-5.625pt)--++(-6.495190528383289pt,0 pt) -- ++(3.2475952641916446pt,5.625pt);
\draw [fill=black] (0.,0.) circle (1.0pt);
\draw[color=black] (0,-0.3) node {$0$};
\draw[color=black] (-2,-0.4) node {$\Sigma_{\mathcal{D}}(0)$};
\draw [fill=black,shift={(4.,0.)},rotate=270] (0,0) ++(0 pt,3.75pt) -- ++(3.2475952641916446pt,-5.625pt)--++(-6.495190528383289pt,0 pt) -- ++(3.2475952641916446pt,5.625pt);
\draw[color=black] (4.166438738303165,0.38559990184332893) node {$x_1$};
\draw[color=black] (2.,-0.4) node {$\Sigma_{\mathcal{N}}(0)$};
\draw [fill=black,shift={(7.,0.)},rotate=90] (0,0) ++(0 pt,3.75pt) -- ++(3.2475952641916446pt,-5.625pt)--++(-6.495190528383289pt,0 pt) -- ++(3.2475952641916446pt,5.625pt);
\draw [fill=black] (11,0.) circle (0.5pt);
\draw[color=black] (11,-0.3) node {$0$};
\draw [fill=black,shift={(15.,0.)},rotate=270] (0,0) ++(0 pt,3.75pt) -- ++(3.2475952641916446pt,-5.625pt)--++(-6.495190528383289pt,0 pt) -- ++(3.2475952641916446pt,5.625pt);
\draw[color=black] (15.166760756965113,0.38559990184332893) node {$x_1$};
\draw [fill=black] (8.,-1.) circle (0.5pt);
\draw[color=black] (8.120747041595465,-1.3) node {$2\rho$};
\draw [fill=black,shift={(15.,-1.)},rotate=270] (0,0) ++(0 pt,3.75pt) -- ++(3.2475952641916446pt,-5.625pt)--++(-6.495190528383289pt,0 pt) -- ++(3.2475952641916446pt,5.625pt);
\draw [fill=black,shift={(7.,-1.)},rotate=90] (0,0) ++(0 pt,3.75pt) -- ++(3.2475952641916446pt,-5.625pt)--++(-6.495190528383289pt,0 pt) -- ++(3.2475952641916446pt,5.625pt);
\draw [fill=black,shift={(0.,2.)}] (0,0) ++(0 pt,3.75pt) -- ++(3.2475952641916446pt,-5.625pt)--++(-6.495190528383289pt,0 pt) -- ++(3.2475952641916446pt,5.625pt);
\draw [fill=black,shift={(11.,2.)}] (0,0) ++(0 pt,3.75pt) -- ++(3.2475952641916446pt,-5.625pt)--++(-6.495190528383289pt,0 pt) -- ++(3.2475952641916446pt,5.625pt);
\draw [fill=black] (11,-1.) circle (0.5pt);
\draw[color=black] (11,-1.3) node {$0$};
\draw [fill=black] (9.5,-1.) circle (0.5pt);
\draw[color=black] (9.623896634207656,-1.3) node {$\rho$};
\draw[color=black] (15.4,-0.2) node {$v$};
\draw[color=black] (15.5,-1.3) node {$v_{\rho}$};
\end{scriptsize}
\end{tikzpicture}
 \caption{The Kelvin transform centered at $0$ acting on $\Sigma_{\mathcal{D}}(0)$ (doted line) and $\Sigma_{\mathcal{N}}(0)$ for the functions $v$ and $v_{\rho}$.}
\end{center}
\end{figure}
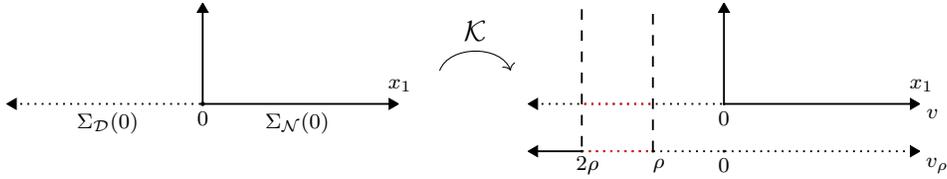
Thus, using the definition of weak solution for the extended problem satisfied by $V$ and $V_{\rho}$ respectively and subtracting those expressions, we obtain
\begin{equation*}
\kappa_s\int_{\Upsilon_\rho^*}y^{1-2s}\nabla(V-V_{\rho})\nabla\varphi\, dxdy=\!\!
\int_{\Upsilon_\rho}\!\!\left(g(|x|^{N-2s}v)v^{\frac{N+2s}{N-2s}}-g(|x^{\rho}|^{N-2s}v_{\rho})v_{\rho}^{\frac{N+2s}{N-2s}}\right)\varphi(x,0)dx.
\end{equation*}
On the other hand,
\begin{align*}
\mkern+40mu\kappa_s&\mkern-55mu\int \limits_{\Upsilon_\rho^*\cap[2\varepsilon\leq|X-O_\rho|\leq\frac{1}{\varepsilon}]}\mkern-50mu y^{1-2s}|\nabla(V-V_{\rho})^+|^2dxdy\leq \kappa_s\int_{\Upsilon_\rho^*}y^{1-2s}|\nabla (\eta_{\varepsilon}(V-V_{\rho})^{+}) |^2dxdy
\\
&=\kappa_s\int_{\Upsilon_\rho^*}y^{1-2s}\nabla(V-V_{\rho})\nabla\varphi\, dxdy+\kappa_s\int_{\Upsilon_\rho^*}y^{1-2s}[(V-V_{\rho})^+]^2|\nabla\eta_{\varepsilon}|^2dxdy
\\
&= \kappa_s\int_{\Upsilon_\rho^*}y^{1-2s}\nabla(V-V_{\rho})\nabla\varphi\, dxdy+I_{\varepsilon}\\
&=\int_{\Upsilon_\rho}\left(g(|x|^{N-2s}v)v^{\frac{N+2s}{N-2s}}-g(|x^{\rho}|^{N-2s}v_{\rho})v_{\rho}^{\frac{N+2s}{N-2s}}\right)\varphi(x,0)dx+I_{\varepsilon}.
\end{align*}
Since $g$ is a nonincreasing function, $|x|\geq|x^{\rho}|$ in $\Upsilon_\rho$ and $v\geq v_{\rho}$ in the set where $\varphi(\cdot,0)\neq0$, it follows that $-g(|x^{\rho}|^{N-2s}v_{\rho})\leq-g(|x|^{N-2s}v)$ and therefore,
\begin{equation}\label{DG1}
\begin{split}
\mkern+40mu\kappa_s\mkern-55mu\int\limits_{\Upsilon_\rho^*\cap[2\varepsilon\leq|X-O_{\rho}|\leq\frac{1}{\varepsilon}]}\mkern-50muy^{1-2s}|\nabla(V-V_{\rho})^+|^2dxdy &\leq
\int_{\Upsilon_\rho}g(|x|^{N-2s}v)\left(v^{\frac{N+2s}{N-2s}}-v_{\rho}^{\frac{N+2s}{N-2s}}\right)\varphi(x,0)dx+I_{\varepsilon}\\
&\leq\int_{\mathcal{A}_{\rho}}g(|x|^{N-2s}v)\left(v^{\frac{N+2s}{N-2s}}-v_{\rho}^{\frac{N+2s}{N-2s}}\right)\varphi(x,0)dx+I_{\varepsilon}.
\end{split}
\end{equation}
Now, if $0\leq v_{\rho}\leq v$ from the Mean Value Theorem, we find
\begin{equation*}
v^{\frac{N+2s}{N-2s}}-v_{\rho}^{\frac{N+2s}{N-2s}}\leq \frac{N+2s}{N-2s} v^{\frac{4s}{N-2s}}(v-v_{\rho}).
\end{equation*}
Now using that $f(t)=t^r$ with $1<r<\frac{N+2s}{N-2s}$, it follows that
$$g(t)t^{\frac{4s}{N-2s}}=\frac{f(t)}{t^{\frac{N+2s}{N-2s}}}t^{\frac{4s}{N-2s}}=\frac{f(t)}{t}=t^{r-1},$$
and $g(t)t^{\frac{4s}{N-2s}}$ is bounded in any interval $(0,t_0)$. Moreover, since $|x|^{N-2s}v(x)=u\left(\frac{x}{|x|^2}\right)$ is bounded from above for $x\in\Upsilon_{\rho}$ and $\rho<0$, we conclude
\begin{align*}
g(|x|^{N-2s}v)\left(v^{\frac{N+2s}{N-2s}}-v_{\rho}^{\frac{N+2s}{N-2s}}\right)&\leq \frac{N+2s}{N-2s} g(|x|^{N-2s}v) v^{\frac{4s}{N-2s}}(v-v_{\rho})\\
&\leq \frac{N+2s}{N-2s} \frac{g(|x|^{N-2s}v) (|x|^{N-2s}v)^{\frac{4s}{N-2s}}}{|x|^{4s}}(v-v_{\rho})\\
&\leq \widetilde{C}_{\rho} \frac{1}{|x|^{4s}}(v-v_{\rho}),
\end{align*}
for a positive constant $\widetilde{C}_{\rho}$ increasing in $\rho$. Then, inequality \eqref{DG1} takes the form
\begin{align*}
\mkern+40mu\kappa_s\mkern-55mu\int\limits_{\Upsilon_{\rho}^*\cap[2\varepsilon\leq|X-O_{\rho}|\leq\frac{1}{\varepsilon}]} \mkern-50mu y^{1-2s}|\nabla(V-V_{\rho})^+|^2dxdy
&\leq \widetilde{C}_{\rho}\int_{\mathcal{A}_{\rho}}\frac{1}{|x|^{4s}}(v-v_{\rho})\varphi(x,0)dx+I_{\varepsilon}
\\
&\leq \widetilde{C}_{\rho}\int_{\mathcal{A}_{\rho}}\frac{1}{|x|^{4s}}\eta_\varepsilon^{2}(x,0)[(v-v_{\rho})^+]^2dx+I_{\varepsilon}
\\
&\leq \widetilde{C}_{\rho}\int_{\mathcal{A}_{\rho}}\frac{1}{|x|^{4s}}[(v-v_{\rho})^+]^2dx+I_{\varepsilon}.
\end{align*}
Using H\"older's inequality with $p=\frac{N}{2s}$ and $q=\frac{2_s^{*}}{2}$ we conclude
\begin{equation*}
\mkern+40mu\kappa_s\mkern-55mu\int \limits_{\Upsilon_{\rho}^*\cap[2\varepsilon\leq|X-O_{\rho}|\leq\frac{1}{\varepsilon}]}\mkern-50muy^{1-2s}|\nabla(V-V_{\rho})^+|^2dxdy \leq \widetilde{C}_{\rho}\left(\int_{\mathcal{A}_{\rho}}\frac{1}{|x|^{2N}}dx\right)^{\frac{2s}{N}}\left(\int_{\Upsilon_{\rho}}[(v-v_{\rho})^+]^{2_s^{*}}dx\right)^{\frac{2}{2_s^{*}}}\!\!\!+ I_\varepsilon.
\end{equation*}
Next, we focus on the term $\displaystyle I_{\varepsilon}=\int_{\Upsilon_\rho^*}y^{1-2s}[(V-V_{\rho})^+]^2|\nabla\eta_{\varepsilon}|^2dxdy$. Define the set
\begin{equation*}
\mathcal{W}_{\varepsilon}=\left\{X\in \Upsilon_{\rho}^*:\ \varepsilon<|X-O_{\rho}|<2\varepsilon \textrm{ or } \frac{1}{\varepsilon}<|X-O_{\rho}|<\frac{2}{\varepsilon}\right\},
\end{equation*}
so that $supp(|\nabla\eta_{\epsilon}|^2)\subseteq \overline{\mathcal{W}_{\varepsilon}}$.
Since $\bigg||\nabla\eta_{\varepsilon}|^{N+1} \chi_{\mathcal{W}_{\varepsilon}}\bigg|\leq c(\frac{1}{\varepsilon^{N+1}}\varepsilon^{N+1}+\varepsilon^{N+1}\frac{1}{\varepsilon^{N+1}})=c'$ and $(V-V_{\rho})^+\in L^{\overline{2}^*}(\Upsilon_{\rho}^*,y^{1-2s} dxdy)$, applying H\"older's inequality with $p=\frac{N+1}{2}$ and $q=\frac{\overline{2}^*}{2}$, we find
\begin{equation*}
\begin{split}
I_{\varepsilon}&\leq\left(\int_{\mathcal{W}_{\varepsilon}}y^{1-2s}[(V-V_{\rho})^+]^{\overline{2}^*}dxdy\right)^{\frac{2}{\overline{2}^*}}\left(\int_{\mathcal{W}_{\varepsilon}}y^{1-2s}|\nabla\eta_{\varepsilon}|^{N+1}dxdy\right)^{\frac{2}{N+1}}
\\
&\leq C\left(\int_{\mathcal{W}_{\varepsilon}}y^{1-2s}[(V-V_{\rho})^+]^{\overline{2}^*}dxdy\right)^{\frac{2}{\overline{2}^*}}\rightarrow 0 \textrm{ as $\varepsilon\rightarrow 0$}.
\end{split}
\end{equation*}
Therefore, applying the trace inequality \eqref{eq:traceineq}, we conclude
\begin{equation*}
\begin{split}
\int_{\Upsilon_{\rho}^*}y^{1-2s}  |\nabla(V-V_{\rho})^+|^2dxdy
&\leq \kappa_s^{-1}\widetilde{C}_{\rho}\left(\int_{\mathcal{A}_{\rho}}\frac{1}{|x|^{2N}}dx\right)^{\frac{2s}{N}}\left(\int_{\Upsilon_{\rho}}[(v-v_{\rho})^+]^{2_s^{*}}dx\right)^{\frac{2}{2_s^{*}}}
\\
&\leq C_{\rho}\left(\int_{\mathcal{A}_{\rho}}\frac{1}{|x|^{2N}}dx\right)^{\frac{2s}{N}}\int_{\Upsilon_{\rho}^*}y^{1-2s}|\nabla(V-V_{\rho})^+|^2dxdy,
\end{split}
\end{equation*}
for a positive constant $C_{\rho}$ increasing with respect to $\rho$.
\end{proof}

\begin{proof}[Proof of Theorem \ref{thmonotonia}]
The proof follows the lines of \cite[Proposition 2.1]{DG} adapted to our framework. First, we establish a starting plane that delimits a hyperspace in which the monotonicity in the $x_1$-direction holds. Next we   extend to such a region progressively until we reach the half-space, and in a second step, to the whole  space having a special care to the singularity of the Kelvin transform at the origin. Since
\begin{equation*}
\int_{\mathcal{A}_{\rho}}\frac{1}{|x|^{2N}}dx\leq\int_{\Upsilon_{\rho}}\frac{1}{|x|^{2N}}dx\rightarrow 0,\ \text{as}\ \rho\rightarrow-\infty,
\end{equation*}
then there exists $-\infty<\rho_0<0$ such that
\begin{equation*}
C_{\rho}\left(\int_{\mathcal{A}_{\rho}}\frac{1}{|x|^{2N}}dx\right)^{\frac{2s}{N}}<1,\quad\mbox{for all}\quad \rho\in (-\infty,\rho_0).
\end{equation*}
From   \eqref{ineq:lemmaDG} we deduce that $(V-V_{\rho})^+\equiv0$ in $\Upsilon_{\rho}^*$, and therefore $V\leq V_{\rho}$ in $\Upsilon_{\rho}^*$ for all $\rho\in (-\infty,\rho_0)$. Consequently $v\leq v_{\rho}$ in $\Upsilon_{\rho}$ for any $\rho\in (-\infty,\rho_0)$.

Assume now that $\rho_0<0$ is maximal. By the Maximum Principle, $v<v_{\rho_0}$ in $\Upsilon_{\rho_0}$. Then $\chi_{\mathcal{A}_{\rho}}\cdot\frac{1}{|x|^{2N}}\rightarrow 0$ point-wisely as $\rho\rightarrow\rho_0$ in $\mathbb{R}_+^N\backslash\{T_{\rho_0}\cup \{O_{\rho_0}\}\}$.

Thus, if $\rho<\rho_0+\delta<0$ then $\chi_{\mathcal{A}_{\rho}}\cdot\frac{1}{|x|^{2N}}\leq \chi_{\Upsilon_{\rho_0+\delta}}\cdot\frac{1}{|x|^{2N}}\in L^1(\mathbb{R}^N_+)$ so that applying the Dominated Convergence Theorem
\begin{equation*}
\int_{\mathcal{A}_{\rho}}\frac{1}{|x|^{2N}}dx\to 0,\quad\mbox{ as }\rho\to\rho_0,
\end{equation*}
and we conclude
\begin{equation*}
C_{\rho}\left(\int_{\mathcal{A}_{\rho}}\frac{1}{|x|^{2N}}dx\right)^{\frac{2s}{N}}<1, \ \forall\rho\in(\rho_0,\rho_0+\delta),
\end{equation*}
for some $\delta>0$ sufficiently small. Therefore $(V-V_{\rho})^+\equiv0$ in $\Upsilon_{\rho}^*$ for $\rho\in(\rho_0,\rho_0+\delta)$ in contradiction with the maximality of $\rho_0$. As a consequence $V<V_{\rho}$ in $\Upsilon_{\rho}^*$ provided $\rho<0$ and by continuity $V\leq V_0$ in $\Upsilon_0^*$, so that $v\leq v_{0}$ in $\Upsilon_0$. Noticing that $|x|=|x^{\rho}|$ for $\rho=0$ we conclude $u\leq u_0$ in $\Upsilon_0$.\newline
The above argument works for the Kelvin transform centered at a point $P=P_{\mu}=(\mu,0,\ldots,0)\in \mathbb{R}_{+}^N$, namely, $v^{\mu}(x)=\frac{1}{|x|^{N-2s}}u(P_{\mu}+\frac{x}{|x|^2})$ with $\mu\leq 0$ (see Figure \ref{fig:kelvinmu}).
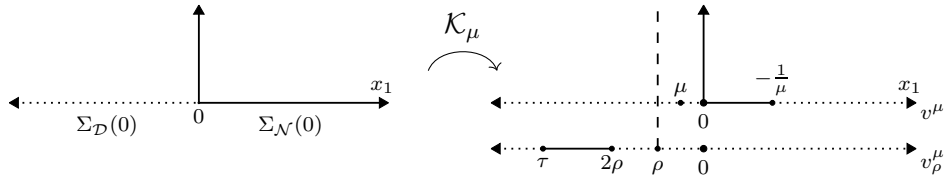
\begin{figure}[!ht]
\begin{center}
\begin{tikzpicture}[scale=0.61][line cap=round,line join=round,>=triangle 45,x=1.0cm,y=1.0cm]
\path[->, bend left=70]  (5,0.7) edge node[auto] {$\mathcal{K}_{\mu}$} (6.5,0.7); 
\draw [line width=0.7pt,dotted] (-4.,0.)-- (0.,0.);
\draw [line width=0.7pt] (0.,0.)-- (4.,0.);
\draw [line width=0.7pt] (0.,0.)-- (0.,2.);
\draw [line width=0.7pt,dotted] (6.5,0.)-- (11.,0.);
\draw [line width=0.7pt] (11.,0.)-- (11.,2.);
\draw [line width=0.7pt,dotted] (11.,-1.)-- (15.5,-1.);
\draw [line width=0.7pt] (11.,0.)-- (12.5,0.);
\draw [line width=0.7pt,dotted] (12.5,0.)-- (15.5,0.);
\draw [line width=0.7pt,dotted] (6.5,-1.)-- (7.5,-1.);
\draw [line width=0.7pt] (7.5,-1.)-- (9.,-1.);
\draw [line width=0.7pt,dotted] (9.,-1.)-- (11.,-1.);
\draw [line width=0.7pt,dash pattern=on 4pt off 4pt] (10.,-1.)-- (10.,2.);
\begin{scriptsize}
\draw [fill=black,shift={(-4.,0.)},rotate=90] (0,0) ++(0 pt,3.75pt) -- ++(3.2475952641916446pt,-5.625pt)--++(-6.495190528383289pt,0 pt) -- ++(3.2475952641916446pt,5.625pt);
\draw [fill=black] (0,0) circle (0.5pt);
\draw[color=black] (0,-0.3) node {$0$};
\draw[color=black] (-2,-0.5) node {$\Sigma_{\mathcal{D}}(0)$};
\draw [fill=black,shift={(4.,0.)},rotate=270] (0,0) ++(0 pt,3.75pt) -- ++(3.2475952641916446pt,-5.625pt)--++(-6.495190528383289pt,0 pt) -- ++(3.2475952641916446pt,5.625pt);
\draw[color=black] (4,0.35) node {$x_1$};
\draw[color=black] (2,-0.5) node {$\Sigma_{\mathcal{N}}(0)$};
\draw [fill=black,shift={(0.,2.)}] (0,0) ++(0 pt,3.75pt) -- ++(3.2475952641916446pt,-5.625pt)--++(-6.495190528383289pt,0 pt) -- ++(3.2475952641916446pt,5.625pt);
\draw [fill=black,shift={(6.5,0.)},rotate=90] (0,0) ++(0 pt,3.75pt) -- ++(3.2475952641916446pt,-5.625pt)--++(-6.495190528383289pt,0 pt) -- ++(3.2475952641916446pt,5.625pt);
\draw [fill=black] (11,0) circle (2.0pt);
\draw[color=black] (11,-0.4) node {$0$};
\draw [fill=black,shift={(15.5,0.)},rotate=270] (0,0) ++(0 pt,3.75pt) -- ++(3.2475952641916446pt,-5.625pt)--++(-6.495190528383289pt,0 pt) -- ++(3.2475952641916446pt,5.625pt);
\draw[color=black] (15.5,0.35) node {$x_1$};
\draw [fill=black,shift={(11.,2.)}] (0,0) ++(0 pt,3.75pt) -- ++(3.2475952641916446pt,-5.625pt)--++(-6.495190528383289pt,0 pt) -- ++(3.2475952641916446pt,5.625pt);
\draw [fill=black,shift={(6.5,-1.)},rotate=90] (0,0) ++(0 pt,3.75pt) -- ++(3.2475952641916446pt,-5.625pt)--++(-6.495190528383289pt,0 pt) -- ++(3.2475952641916446pt,5.625pt);
\draw [fill=black] (11,-1) circle (2.0pt);
\draw[color=black] (11,-1.4) node {$0$};
\draw [fill=black,shift={(15.5,-1.)},rotate=270] (0,0) ++(0 pt,3.75pt) -- ++(3.2475952641916446pt,-5.625pt)--++(-6.495190528383289pt,0 pt) -- ++(3.2475952641916446pt,5.625pt);
\draw [fill=black] (10.5,0.) circle (1.5pt);
\draw[color=black] (10.5,0.3) node {$\mu$};
\draw [fill=black] (12.5,0.) circle (1.5pt);
\draw[color=black] (12.5,0.5) node {{\scriptsize $-\frac{1}{\mu}$}};
\draw [fill=black] (7.5,-1.) circle (1.5pt);
\draw[color=black] (7.5,-1.3) node {$\tau$};
\draw [fill=black] (9.,-1.) circle (1.5pt);
\draw[color=black] (9,-1.35) node {$2\rho$};
\draw [fill=black] (10.,-1.) circle (1.5pt);
\draw[color=black] (10.,-1.4) node {$\rho$};
\draw[color=black] (16,-0.2) node {$v^{\mu}$};
\draw[color=black] (16,-1.3) node {$v_{\rho}^{\mu}$};
\end{scriptsize}
\end{tikzpicture}
  \caption{The Kelvin transform centered at $\mu\leq0$ acting on $\Sigma_{\mathcal{D}}(0)$ (doted line) and $\Sigma_{\mathcal{N}}(0)$ for the functions $v^{\mu}$ and $v_{\rho}^{\mu}$. The set $\Sigma_{\mathcal{N}}(0)$ is transformed into those $x\in\mathbb{R}_+^N$ such that $0<x_1<-\frac{1}{\mu}$, so $v_{\rho}^{\mu}$ satisfies a Neumann condition on $\tau<x_1<2\rho$ with $\tau=2\rho+\frac{1}{\mu}$. }\label{fig:kelvinmu}
\end{center}
\end{figure}
This centered fractional Kelvin transform $v^{\mu}$ satisfies a Dirichlet condition in the part of the boundary with
$x_N=0$ and $x_1<0$ so we can prove as before that for any $\rho<0$ the inequality $v^{\mu}\leq v_{\rho}^{\mu}$ holds in $\Upsilon_{\rho}$.
Since $\rho<0$ is arbitrary, it follows that $v^{\mu}\leq v_{0}^{\mu}$ in $\Upsilon_{0}$. Thus $u\leq u_{\mu}$ in $\Upsilon_{\mu}$ for $\mu\leq0$, so $u$ is nondecreasing in the $x_1$-direction provided $x_1<0$.

Now we extend progressively the region in which the monotonicity holds reaching $\Upsilon_{\mu}$ for $\mu>0$.
First, observe that we cannot continue as before due to the singularity of the Kelvin transform at the origin: we cannot take a moving plane starting at $\rho=-\infty$ since for $\rho$ large there are points where the Neumann boundary condition holds (and the solution is positive) which are reflected to the Dirichlet part of the boundary. In terms of the test functions, for $\rho$ large enough  the function $\left(V-V_{\rho}\right)^+$ is not allowed to be chosen as  test function for the problem satisfied by the reflected function $V_{\rho}$ , since it does not vanish at those  points of the boundary where the Dirichlet condition for $V_{\rho}$ holds.

Nevertheless, an inequality similar to \eqref{ineq:lemmaDG} holds for $(v^{\mu}-v_{\rho}^{\mu})^+$ if $\rho$ is close to $0$ so that we extend the inequality $v^{\mu}(x)<v_{\rho}^{\mu}(x)=v^{\mu}(x_{\rho})$ for every $\rho<0$ fixed, moving $\mu$ from $\mu=0$ where the strict inequality is true up to $\mu=\frac{-1}{2\rho}$.\newline
If $\mu\geq 0$, the fractional Kelvin transform centered at the point $P_{\mu}$  (denoted by $v^{\mu}(x)$)   satisfies a Dirichlet boundary condition at points $x\in\mathbb{R}_{+}^{N}$ with $x_N=0$ and $\frac{-1}{\mu}<x_1<0$ ($x_1<0$ if $\mu=0$ as in the previous step) and a Neumann condition on the remaining part of the boundary. Then, if $-\frac{1}{2\mu}<\rho<0$ it follows that $V^{\mu}$, and hence $(V^{\mu}-V_{\rho}^{\mu})^+$, vanishes where the Dirichlet condition holds for $V^{\mu}$ and  also where the Dirichlet condition holds for the reflected function $V_{\rho}^{\mu}$ (therefore $\varphi_{\varepsilon}$ is an allowed test function).
\begin{figure}[!ht]
\begin{center}
\begin{tikzpicture}[scale=0.69][line cap=round,line join=round,>=triangle 45,x=1.0cm,y=1.0cm]
\path[->, bend left=70]  (4.5,0.7) edge node[auto] {$\mathcal{K}_{\mu}$} (6,0.7);
\draw [line width=0.7pt,dotted] (-4.,0.)-- (0.,0.);
\draw [line width=0.7pt] (0.,0.)-- (4.,0.);
\draw [line width=0.7pt] (0.,0.)-- (0.,2.);
\draw [line width=0.7pt] (11.,0.)-- (13.,0.);
\draw [line width=0.7pt] (11.,0.)-- (11.,2.);
\draw [line width=0.7pt] (6.,0.)-- (7.,0.);
\draw [line width=0.7pt,dotted] (7.,0.)-- (11.,0.);
\draw [line width=0.7pt] (8,-1.)-- (6.,-1.);
\draw [line width=0.7pt,dotted] (8,-1.)-- (12.,-1.);
\draw [line width=0.7pt] (12.,-1.)-- (13.,-1.);
\draw [line width=0.7pt,dash pattern=on 4pt off 4pt] (9.5,-1)-- (9.5,2.);
\begin{scriptsize}
\draw [fill=black,shift={(-4.,0.)},rotate=90] (0,0) ++(0 pt,3.75pt) -- ++(3.2475952641916446pt,-5.625pt)--++(-6.495190528383289pt,0 pt) -- ++(3.2475952641916446pt,5.625pt);
\draw [fill=black] (0,0) circle (0.5pt);
\draw[color=black] (0,-0.3) node {$0$};
\draw[color=black] (-2,-0.4) node {$\Sigma_{\mathcal{D}}(0)$};
\draw[color=black] (2,-0.4) node {$\Sigma_{\mathcal{N}}(0)$};
\draw [fill=black,shift={(4.,0.)},rotate=270] (0,0) ++(0 pt,3.75pt) -- ++(3.2475952641916446pt,-5.625pt)--++(-6.495190528383289pt,0 pt) -- ++(3.2475952641916446pt,5.625pt);
\draw[color=black] (4.1,0.3) node {$x_1$};
\draw [fill=black,shift={(0.,2.)}] (0,0) ++(0 pt,3.75pt) -- ++(3.2475952641916446pt,-5.625pt)--++(-6.495190528383289pt,0 pt) -- ++(3.2475952641916446pt,5.625pt);
\draw [fill=black] (11,0) circle (1.5pt);
\draw[color=black] (11,-0.3) node {$0$};
\draw [fill=black,shift={(13.,0.)},rotate=270] (0,0) ++(0 pt,3.75pt) -- ++(3.2475952641916446pt,-5.625pt)--++(-6.495190528383289pt,0 pt) -- ++(3.2475952641916446pt,5.625pt);
\draw[color=black] (13.1,0.3) node {$x_1$};
\draw [fill=black,shift={(11.,2.)}] (0,0) ++(0 pt,3.75pt) -- ++(3.2475952641916446pt,-5.625pt)--++(-6.495190528383289pt,0 pt) -- ++(3.2475952641916446pt,5.625pt);
\draw [fill=black,shift={(6.,0.)},rotate=90] (0,0) ++(0 pt,3.75pt) -- ++(3.2475952641916446pt,-5.625pt)--++(-6.495190528383289pt,0 pt) -- ++(3.2475952641916446pt,5.625pt);
\draw [fill=black] (7.,0.) circle (1.5pt);
\draw[color=black] (7.,0.5) node {{\scriptsize $-\frac{1}{\mu}$}};
\draw [fill=black] (11.25,0.) circle (1.5pt);
\draw[color=black] (11.3,0.3) node {$\mu$};
\draw [fill=black] (8,-1.) circle (1.5pt);
\draw[color=black] (8,-1.3) node {$2\rho$};
\draw [fill=black,shift={(6.,-1.)},rotate=90] (0,0) ++(0 pt,3.75pt) -- ++(3.2475952641916446pt,-5.625pt)--++(-6.495190528383289pt,0 pt) -- ++(3.2475952641916446pt,5.625pt);
\draw [fill=black] (9.,0.) circle (1.5pt);
\draw[color=black] (9.,0.5) node {{\scriptsize $-\frac{1}{2\mu}$}};
\draw [fill=black] (12.,-1.) circle (1.5pt);
\draw[color=black] (12.,-1.3) node {$\tau$};
\draw [fill=black,shift={(13.,-1.)},rotate=270] (0,0) ++(0 pt,3.75pt) -- ++(3.2475952641916446pt,-5.625pt)--++(-6.495190528383289pt,0 pt) -- ++(3.2475952641916446pt,5.625pt);
\draw [fill=black] (9.5,-1.) circle (1.5pt);
\draw[color=black] (9.5,-1.3) node {$\rho$};
\draw [fill=black] (11.,-1.) circle (1.5pt);
\draw[color=black] (11.,-1.3) node {$0$};
\draw[color=black] (13.5,-0.2) node {$v^{\mu}$};
\draw[color=black] (13.5,-1.3) node {$v_{\rho}^{\mu}$};
\end{scriptsize}
\end{tikzpicture}
\caption{The Kelvin transform centered at $\mu\geq0$ acting on $\Sigma_{\mathcal{D}}(0)$ (doted line) and $\Sigma_{\mathcal{N}}(0)$ for the functions $v^{\mu}$ and $v_{\rho}^{\mu}$. The set $\Sigma_{\mathcal{D}}(0)$ is transformed into the $x\in\mathbb{R}_+^N$ such that $x_N=0$ and $-\frac{1}{\mu}<x_1<0$, so the reflected function $v_{\rho}^{\mu}$ satisfies a Dirichlet condition on $2\rho<x_1<\tau$ with $\tau=2\rho+\frac{1}{\mu}$. It follows that for $x\in\Upsilon_{\rho}$ the function $v^{\mu}$ vanish where the Dirichlet condition holds for $v_{\rho}^{\mu}$. }
\end{center}
\end{figure}
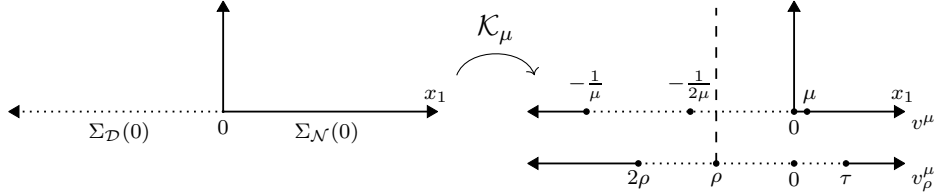
Thus, proceeding exactly as in the case $\mu=0$, we obtain
\begin{equation*}
\int_{\Upsilon_{\rho}^*}y^{1-2s}|\nabla(V^{\mu}-V_{\rho}^{\mu})^+|^2dxdy\leq
 C_{\rho}\left(\int_{\mathcal{A}_{\rho}^{\mu}}\frac{1}{|x|^{2N}}dx\right)^{\frac{2s}{N}}\int_{\Upsilon_{\rho}^*}y^{1-2s}|\nabla(V^{\mu}-V_{\rho}^{\mu})^+|^2dxdy,
\end{equation*}
where $C_{\rho}$ is increasing with respect to $\rho$ and $\mathcal{A}_{\rho}^{\mu}=\{x\in \Upsilon_{\rho}\backslash O_{\rho}: v^{\mu}\geq v_{\rho}^{\mu}\}$.\newline
If we now fix $\rho<0$ the previous estimate holds for any $\mu\in(0,-\frac{1}{2\rho})$ and, since $\frac{1}{|x|^{2N}}\in L^{1}(\Upsilon_{\rho})$, applying the Dominated Convergence Theorem we conclude $\chi_{\mathcal{A}_{\rho}^{\mu}}\cdot\frac{1}{|x|^{2N}}\to 0$ as $\mu\to 0$ in $\mathbb{R}^{N}\backslash\{T_{\rho}\cup P_{\rho}\}$, we recall that $P_{\rho}=(2\rho,0,\ldots,0)$ is the reflected point of the origin, which is the singular point of every transform $V^{\mu}$. As a consequence
\begin{equation*}
C_{\rho}\left(\int_{\mathcal{A}_{\rho}^{\mu}}\frac{1}{|x|^{2N}}dx\right)^{\frac{2s}{N}}<1,
\end{equation*}
for some $\rho_0\in(\frac{-1}{2\mu},0)$ and the monotonicity follows. Finally, suppose that $\mu_0<-\frac{1}{2\rho_0}$ is maximal such that $v^{\mu}\leq v_{\rho}^{\mu}$ in $\Upsilon_{\rho}$ for all $0<\mu<\mu_0$. Then, by the maximum principle, $v^{\mu}< v_{\rho}^{\mu}$ and hence $\mathcal{A}_{\rho}^{\mu}\to\emptyset$ as $\mu\to\mu_0$. Thus, there exists $\epsilon>0$ such that
\begin{equation*}
C_{\rho}\left(\int_{\mathcal{A}_{\rho}^{\mu}}\frac{1}{|x|^{2N}}dx\right)^{\frac{2s}{N}}<1\ \text{for}\ \mu\in(\mu_0,\mu_0+\epsilon).
\end{equation*}
We conclude that $v^{\mu}\leq v_{\rho}^{\mu}$ for $\mu>\mu_0$ and close to $\mu_0$ in contradiction with the maximality of $\mu_0$.
\newline
In sum, for every $\rho<0$ and $\mu\leq-\frac{1}{2\rho}$ we have $v^{\mu}\leq v_{\rho}^{\mu}$ in $\Upsilon_{\rho}$ or, equivalently, fixed $\mu>0$ the inequality holds for every $-\frac{1}{2\mu}<\rho<0$. Letting $\rho\to0$ we get $v^{\mu}\leq v_0^{\mu}$ in $\Upsilon_{0}$, i.e.,
 $v^{\mu}(x_1,x')\leq v^{\mu}(-x_1,x')$
 for all $x$ with $x_1<0$, so that $u\leq u_{\mu}$ in $\Upsilon_{\mu}$ with $\mu>0$. Since $\mu>0$ is arbitrary we get that $u$ is nondecreasing in the $x_1$-direction in whole $\mathbb{R}_{+}^{N}$.
\end{proof}

\begin{remark}
Let us observe that the method described in the above Theorem  in the $x_1$-direction  may be applied to any other direction $x_2,  \ldots, x_{N-1}$, centered at  any point $P$ of the form $P=(0, P_2, \dots, P_{N-1}, 0)$, with a hyperplane orthogonal to both to the $e_1$ and $e_n$ directions.
Thus, due to the arbitrary of the point $P$,  we can deduce that $u$ does not depend to the $x_2, \dots, x_{N-1}$ variables.
\end{remark}
\section{A priori bounds in $L^\infty(\Omega)$.}
In this section we prove Theorem \ref{acotacion} exploiting the   blow-up method by Guidas-Spruck (see \cite{GS}). To this aim  we will make use of the estimates proved in \cite[Theorem 1.1]{CCLO} that guarantee the compactness needed in order to accomplish this limit step.
Then, with the same ideas, we prove Theorem 1.3 using  the uniform estimates proved in \cite[Corollary 1.1]{CCLO} for the moving boundary conditions (as in hypotheses $(B_1)$-$(B_3)$).
\begin{proof}[Proof of Theorem \ref{acotacion}]
We argue by contradiction:  set  $\Lambda>0$ given by Theorem \ref{Thuno} and assume that  there exists sequences $\{\lambda_k\}\subset [0,\Lambda]$, $\{u_k\}$ of solutions to problems $(P_{\lambda_k})$ and $\{p_k\}\subset\overline{\Omega}$ of points verifying
\begin{equation*}
M_{k}= \sup\limits_{x \in \overline{\Omega}} \ u_{k}(x) = u_{k}(p_{k}) \rightarrow+\infty,\quad \text{as}\ k\rightarrow\infty.
\end{equation*}
Let us set $\mu_k=M_k^{-\frac{r-1}{2s}}$ and define the functions $v_k(y)=\frac{1}{M_k}u(p_k+\mu_k y)$. Note that $v_k(y)$ is defined in $\Omega_{k}=\frac{1}{\mu_k}\,\left(\Omega-p_k\right)$ as well as $v_k(0)=1$ and $\|v_k\|_{L^\infty(\Omega_k)}\leq1$ for all $k\geq0$. Moreover, the scaled
function $v_k$ satisfies the problem
\begin{equation*}
        \left\{
        \begin{tabular}{rll}
        $\displaystyle (-\Delta)^sv_k=\lambda_k M_k^{q-r}v_k^q+v_k^r$ \quad  &$v_k>0$, &in $\Omega_{k}=\frac{1}{\mu_k}\left(\Omega-p_k\right)$, \\
        $ v_k=0\mkern+116mu $ & &on $ \Sigma_{\mathcal{D}}^k$,\\
        $\displaystyle\frac{\partial v_k}{\partial \nu}=0\mkern+116mu $ & &on $ \Sigma_{\mathcal{N}}^k$,
        \end{tabular}
        \right.
\end{equation*}
where $\Sigma_{\mathcal{D}}^k$ and $\Sigma_{\mathcal{N}}^k$ are the transformed boundary manifolds.

Now we study the limit problem obtained as $k\to\infty$. To carry out this step we need some compactness properties for the sequence $\{v_k\}$ in order to guarantee the convergence in some sense.
By \cite[Theorem 1.1]{CCLO} the sequence $\{v_k\}$ is uniformly bounded in ${C^\gamma(\overline{\Omega_k})}$ for some $\gamma\in \left(0,\frac12\right)$.
Then, by the Ascoli-Arzel\'a Theorem, there exists a subsequence $\{v_k\}$  uniformly convergent over compact sets in  $\overline{\mathbb{R}_+^N}$
to a function $v\in C^{\eta}(\overline{\mathbb{R}_+^N})$ for some $0<\eta<\gamma<\frac12$. Moreover $\|v\|_{L^\infty(\mathbb{R}^N)}\leq 1$ and $v(0)=1$.\newline
On the other hand, the problem satisfied by the limit function $v$ depends on the position of the point $\displaystyle p=\lim_{k\to\infty}p_k$.  Let us set
\begin{equation*}
d_k^{\mathcal{D}}=dist(p_k,\Sigma_{\mathcal{D}}^k)\quad\mbox{and}\quad d_k^{\mathcal{N}}=dist(p_k,\Sigma_{\mathcal{N}}^k).
\end{equation*}
and define $d_k^{\Omega}=\min\{d_k^{\mathcal{D}},d_k^{\mathcal{N}}\}$. We distinguish several cases according to the behavior of the sequences $\frac{d_k^i}{\mu_k}$ with $i=\Omega, \mathcal{D}, \mathcal{N}$.

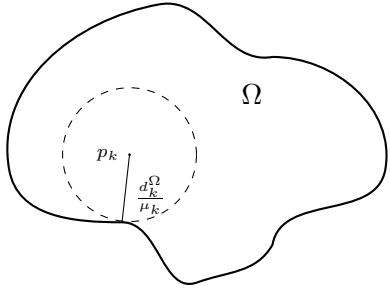
\begin{figure}[!ht]
  \begin{center}
    \begin{tikzpicture}[scale=1]
    \draw [thick] (1,1.5)
      to [out=190,in=10] (-0.5,2.2)
      to [out=190,in=100] (-2.5,0)
      to [out=-80,in=180] (-1,-0.7)
      to [out=0,in=200]   (0,-1.5)
      to [out=20,in=-120] (1,-1)
      to [out=80,in=-100] (2.5,0)
      to [out=80,in=0]  (1,1.5) ;
 \draw [dashed] (-0.9,0.2) circle [radius=0.89];
 \draw [fill] (-0.9,0.2) circle [radius=0.01] node [black,left] {{\tiny $p_k$}};
 \draw [fill] (1,1)
  node [black,left] {{$\Omega$}};
  \draw (-1,-0.7) -- (-0.9,0.2) node [black,right, pos=.4] {{\tiny $\frac{d_k^{\Omega}}{\mu_k}$}};
   \end{tikzpicture}
       \caption{The relevant geometry after dilation of variables lies in a neighbourhood of $p_k$ such as the one of the picture.}\label{fig:dilatacion}
  \end{center}
\end{figure}
1. \textbf{Interior case:} $\left\{\frac{d_k^{\Omega}}{\mu_k}\right\}\rightarrow+\infty$.\newline
Since $B_{d_k^{\Omega}/\mu_k}(0)\subset\Omega_{k}$ (see Figure
\ref{fig:dilatacion}) we have that $\Omega_k\to\mathbb{R}^N$ and the limit function $v$ is a positive bounded solution to
\begin{equation*}
(-\Delta)^s v=v^r \quad\mbox{in }\mathbb{R}^N,
\end{equation*}
Then, by \cite[Theorem 1]{CLO} (see also \cite[Theorem 3.1]{BCdPS})we conclude $v\equiv0$, in contradiction with $v(0)=1$.

2. \textbf{Boundary Cases:} $\left\{\frac{d_k^{\Omega}}{\mu_k}\right\}\rightarrow d^{\Omega}\in \mathbb{R}^+.$\newline
In this situation we have several possibilities:
\begin{itemize}
\item[2.1] \underline{Dirichlet Case}: $\left\{\frac{d_k^{\mathcal{D}}}{\mu_k}\right\}\rightarrow d^{\mathcal{D}} \in \mathbb{R}^+ $ and $\left\{\frac{d_k^{\mathcal{N}}}{\mu_k}\right\}\rightarrow +\infty$.\newline
Now, as $\Sigma_{\mathcal{D}}$ is a $(N-1)$-dimensional smooth manifold, we have that, up to a rotation
\begin{equation*}
\Omega_{k}\to \Omega_{d^{\mathcal{D}}}\equiv \{x\in\mathbb{R}^N:x_N>-d^{\mathcal{D}}\},
\end{equation*}
and the limit function $v$ is a positive solution to
\begin{equation*}
\left\{
   \begin{tabular}{rl}
    $(-\Delta)^s v=v^r$  & in $\Omega_{d^{\mathcal{D}}}\ $,\\
    $v=0\mkern+7.5mu $& in $\{x_N=-d^{\mathcal{D}}\}$,
   \end{tabular}
\right.
\end{equation*}
with  $\|v\|_{L^\infty(\Omega_{d^{\mathcal{D}}})}\leq 1$ and
$v(0)=1$. Thus,  if $d^{\mathcal{D}}=0$ we have a
contradiction with the continuity since $v(0)=1$ while if
$d^{\mathcal{D}}>0$ we have a contradiction with \cite[Theorem 3.4]{BCdPS}

\item[2.2] \underline{Neumann case}: $\left\{\frac{d_k^{\mathcal{D}}}{\mu_k}\right\}\rightarrow +\infty$ and $\left\{\frac{d_k^{\mathcal{N}}}{\mu_k}\right\}\rightarrow d^{\mathcal{N}} \in \mathbb{R}^+ $.\newline
As before, since $\Sigma_{\mathcal{N}}$ is a $(N-1)$-dimensional smooth manifold, we have that, up to rotation,
\begin{equation*}
\Omega_{k}\to\Omega_{d^{\mathcal{N}}}\equiv\{x\in\mathbb{R}^N:x_N>-d^{\mathcal{N}}\},
\end{equation*}
and the limit function $v$ is a positive solution to
\begin{equation*}
\left\{
    \begin{tabular}{rl}
     $(-\Delta)^s v=v^r$  & in $\Omega_{d^{\mathcal{N}}}$, \\
     $\frac{\partial v}{\partial x_N}=0\mkern+8mu $& in $\{x_N=-d^{\mathcal{N}}\},$
    \end{tabular}
\right.
\end{equation*}
with $\|v\|_{L^\infty(\Omega_{d^{\mathcal{N}}})}\leq 1$ and
$v(0)=1$. Then, if we define the translated function
$w(x)=v(x_1,x_2,\ldots,x_N+d^{\mathcal{N}})$ it follows that
\begin{equation*}
\left\{
   \begin{tabular}{rl}
    $(-\Delta)^s w=w^r$  & in $\mathbb{R}_+^N$, \\
    $\frac{\partial w}{\partial x_N}=0\mkern+12mu $& in $\{x_N=0\}$,
   \end{tabular}
\right.
\end{equation*}
with $\|w\|_{L^\infty(\mathbb{R}_+^N)}\leq 1$ and $w(0,0,\ldots,d^{\mathcal{N}})=1$. Extending to the whole space by reflection through the hyperplane $\{x_N=0\}$, thanks to  \cite[Theorem 3.1]{BCdPS}, it follows that $w\equiv0$ and we get a contradiction with $w(0,0,\ldots,d^{\mathcal{N}})=1$.

\item[2.3] \underline{Interphase Case}:  $\left\{\frac{d_k^{\mathcal{D}}}{\mu_k}\right\}\rightarrow d^{\mathcal{D}}\in \mathbb{R}^+ $ and $\left\{\frac{d_k^{\mathcal{N}}}{\mu_k}\right\}\rightarrow d^{\mathcal{N}}\in \mathbb{R}^+$.\newline
Let us set $d^{\Omega}=\min\{d^{\mathcal{D}},d^{\mathcal{N}}\} \geq 0$ and note that $\Sigma_{\mathcal{D}}^k$, $\Sigma_{\mathcal{N}}^k$ and $\Gamma_k=\Sigma_{\mathcal{D}}^k\cap\overline{{\Sigma}_{\mathcal{N}}^k}$ are smooth manifolds by hypotheses $(\mathfrak{B})$. Hence, we can assume that, up to a rotation,
\begin{equation*}
\Omega_{k}\to\Omega_{d^{\Omega}}\equiv\{x\in\mathbb{R}^N:x_N>-d^{\Omega}\},
\end{equation*}
and  the interphase $\Gamma_{k}\to\{x_1=\tau\}$ for some finite 
$\tau\in\mathbb{R}$. Then   the limit function $v$ is a positive solution to
\begin{equation*}
\left\{
   \begin{tabular}{rl}
    $(-\Delta)^s v=v^r$  & in $\Omega_{d^{\Omega}}$, \\
    $v=0\mkern+7.5mu $& in $\{x_N=-d^{\Omega}\}\cap\{x_1\leq \tau\},$\\
    $\frac{\partial v}{\partial x_N}=0\mkern+7.5mu $& in $\{x_N=-d^{\Omega}\}\cap\{x_1>\tau\},$
   \end{tabular}
\right.
\end{equation*}
with $\|v\|_{L^\infty(\Omega_{d^\Omega})}\leq1$ and $v(0)=1$.
\begin{itemize}
\item[1)] If $d^{\Omega}=0$ and $\tau\geq0$ we get a contradiction with the continuity of $v$, since the maximum is achieved at a point on the Dirichlet boundary where $v\equiv0$.

\item[2)] If $d^{\Omega}>0$ and $\tau\geq0$ we get a contradiction with the monotonicity (Theorem~\ref{thmonotonia}) and the Hopf Lemma at the   maximum  point. Indeed  it is sufficient  to have   the monotonicity of the solution $v$ with respect to   the $x_1$-direction up to $x_1=\tau$.

\item[3)] If   $\tau<0$, we reach, once again, a contradiction with the monotonicity and the Hopf Lemma at the point of maximum. In this step it is necessary to use the monotonicity of $v$ with respect to the  $x_1$-direction in the whole space.
\end{itemize}
\end{itemize}
\end{proof}

With the same ideas, we can prove the next result concerning the moving boundary conditions.

\begin{proof}[Proof of Theorem \ref{acotacion_unif}]
As we did in Theorem \ref{acotacion}, we argue by contradiction. Assume that there exists a sequence
$\{u_{\alpha}\}_{\alpha\in I_{\varepsilon}}$ of solutions to problems $(P_{\alpha,\lambda})$,
a sequence of points $\{p_\alpha\}\subset\overline{\Omega}$, $\overline{\alpha}\in I_{\varepsilon}$
and a sequence of numbers $\mu_{\alpha}=M_{\alpha}^{\frac{1-r}{2s}}$ verifying
\begin{equation*}
M_{\alpha}= \sup\limits_{x \in \overline{\Omega}} \ u_{\alpha}(x) = u_{\alpha}(p_{\alpha}) \rightarrow+\infty,\ \text{as}\ \alpha\rightarrow\overline{\alpha}\,.
\end{equation*}

We have to distinguish  several cases. The interior, Dirichlet and Neumann cases can be proved following  the corresponding cases in  Theorem \ref{acotacion}.

As far as the interface case is concerned,  we need some compactness for the sequence $\{u_{\alpha}\}$ as $\alpha\to\overline{\alpha}$. Since we are considering sets $\Sigma_{\mathcal{D}}(\alpha)$ with $\alpha\in I_{\varepsilon}=[\varepsilon,|\partial\Omega|]$ for some $\varepsilon>0$ and satisfying hypotheses $(\mathfrak{B}_{\alpha})$ and $(B_1)$-$(B_3)$, by  \cite[Corollary 1.1]{CCLO} the sequence $\{u_{\alpha}\}$ is uniformly bounded in $C^\gamma(\overline{\Omega})$ for some $\gamma\in \left(0,\frac12\right)$ and so the conclusion follows as in the corresponding case in Theorem \ref{acotacion}.
\end{proof}
\section{Minimal and mountain-pass solutions}
We devote this section to the proof of Theorem \ref{Thuno}. To do so, we make full use of the extension technique. We recall that in terms of the $s$-extension, problem \eqref{problema_abajo} can be reformulated as
\begin{equation}\label{eq:funcional}
        \left\{
        \begin{array}{rlcl}
        \displaystyle -\text{div}(y^{1-2s}\nabla U)&\!\!\!\!=0  & & \mbox{ in } \mathcal{C}_{\Omega} , \\
        \displaystyle B(U)&\!\!\!\!=0   & & \mbox{ on } \partial_L\mathcal{C}_{\Omega} , \\
        U&\!\!\!\!>0   & &\mbox{ on } \Omega\times\{y=0\},
        \\
         \displaystyle \frac{\partial U}{\partial \nu^s}&\!\!\!\!= f_\lambda(U) & &  \mbox{ on } \Omega\times\{y=0\},
         \end{array}
         \right.
         \tag{$P_{\lambda}^*$}
\end{equation}
where $f_\lambda(s)=\lambda |s|^{q-1}s+|s|^{r-1}s$. Associated to the problem \eqref{eq:funcional} we consider the Euler-Lagrange functional $J_{\lambda}: H_{\Sigma^*_{\mathcal{D}}}^1(\mathcal{C}_{\Omega},y^{1-2s}dxdy)\to \mathbb{R}$ given by
\begin{equation*}
J_{\lambda}(U)=\frac{\kappa_s}{2}\int_{\mathcal{C}_{\Omega}}y^{1-2s}|\nabla U|^2dxdy-\int_{\Omega}F_{\lambda}(U(x,0))dx,
\end{equation*}
where $F_{\lambda}(s)\equiv\int_{0}^{s}f_{\lambda}(\tau)d\tau$.
Although $J_\lambda$ does not satisfies the Palais-Smale (PS for short) condition,
due to the unboundedness of the cylinder $\mathcal{C}_{\Omega}$, we show the PS condition for the functional $I_\lambda$.
\begin{lemma}\label{PS}
Let $\{u_n\}\subset H_{\Sigma_{\mathcal{D}}}^s(\Omega)$ be a PS sequence, i.e., $I_\lambda(u_n)\to c$ and $I_\lambda '(u_n)\to 0$. Then, there exist a subsequence
(again denoted by) $u_n$ strongly convergent in $H_{\Sigma_{\mathcal{D}}}^s(\Omega)$.
\end{lemma}
\begin{proof}
Since $I_\lambda(u_n)\to c$ we have that $\| u_n \|_{H_{\Sigma_{\mathcal{D}}}^s(\Omega)}\le C$ uniformly for some positive constant.
By the Sobolev embeddings, there exists a subsequence still denoted by $\{u_n\}$ such that
\begin{equation}\label{PS1}
u_n\to u \quad \mbox{in }L^r(\Omega), \mbox{ for any } 1\le r<2^*_s,
\end{equation}
and
\begin{equation}\label{PS2}
u_n\rightharpoonup u \quad \mbox{in } H_{\Sigma_{\mathcal{D}}}^s(\Omega).
\end{equation}
Using that $I_\lambda'(u_n)\to 0$ together with \eqref{PS1}-\eqref{PS2}, we have the strong convergence proving the PS condition.
\end{proof}
\begin{proof}[Proof of Theorem \ref{Thuno}-(1)]
Consider the eigenvalue problem associated to the first eigenvalue
$\lambda_1^s$, and let $\varphi_1$ be the positive normalized in $L^2(\Omega)$ associated eigenfunction.
Using $\varphi_1$ as a test function in problem
\eqref{problema_abajo}, we have
\begin{equation*}
(\lambda_1^s-\lambda)\int_{\Omega}u\varphi_1dx=\int_{\Omega}u^r\varphi_1dx,
\end{equation*}
and hence necessarily  $\lambda<\lambda_1^s$. On the other hand,
using the fractional Sobolev inequality together with Poincar\'e
inequality we find
\begin{equation*}
\begin{aligned}
 I_{\lambda}(v)&=\frac 12\int_{\Omega}|(-\Delta)^{s/2}v|^2dx-\frac{\lambda}{2}\int_{\Omega} |v|^{2}dx-\frac{1}{r+1}\int_{\Omega}|v|^{r+1} dx \\
               &\geq c_1\left(1-\frac{\lambda}{\lambda_1^s}\right)\int_{\Omega}|(-\Delta)^{s/2}v|^2dx-c_2{\left( \int_{\Omega}|(-\Delta)^{s/2}v|^2dx \right)}^{(r+1)/2},
\end{aligned}
\end{equation*}
for positive constants $c_1,\,c_2$. Therefore, $v=0$ is a local
minimum for $I_\lambda$ and, since $I_\lambda(t v)\to -\infty$ as
$t\to \infty$, the functional $I_{\lambda}$ satisfies the
hypotheses of the Mountain Pass Theorem by Ambrosetti-Rabinowitz
\cite{AR}. Hence, by Lemma \ref{PS}, we obtain the existence of at
least one solution for $0<\lambda<\lambda_1^s$. Even more, the
bifurcation result is a consequence of the classical Rabinowitz
Theorem \cite{R}.
\end{proof}

Next, in order to continue with the proof of Theorem \ref{Thuno}, we establish some preliminary results. Some of these results can be proved for more general nonlinearities $f(u)$, with $f$ at least
continuous,  satisfying the growth condition $0\leq f(s)\leq c(1+|s|^p)$ for some $p>0$. In such  cases we will denote the associated extension problem as $(P_f^*)$.

The first result deals with the sub and supersolutions method, the proof is rather standard and so we omit it.

\begin{lemma}\label{existencia}
Suppose that there exist a subsolution $U_1$ and a supersolution $U_2$ to $(P_f^*)$, i.e., $U_1,U_2 \in H_{\Sigma_{\mathcal{D}}^*}^1(\mathcal{C}_{\Omega},y^{1-2s}dxdy)$ such that $B(U_1)\leq 0$, $B(U_2)\geq 0$ on $\partial_L \mathcal{C}_\Omega$ and for
every nonnegative $\phi \in H_{\Sigma_{\mathcal{D}}^*}^1(\mathcal{C}_{\Omega},y^{1-2s}dxdy)$ the following inequalities are satisfied:
\begin{equation*}
\begin{split}
\kappa_s\int_{\mathcal{C}_{\Omega}}y^{1-2s}\nabla U_1\nabla \phi dxdy&\leq\int_{\Omega}f(U_1(x,0))\phi (x,0)dx\\
\kappa_s\int_{\mathcal{C}_{\Omega}}y^{1-2s}\nabla U_2\nabla \phi dxdy&\geq\int_{\Omega}f(U_2(x,0))\phi (x,0) dx\,,
\end{split}
\end{equation*}
respectively.
Assume moreover that $U_1\leq U_2$ in $\mathcal{C}_{\Omega}$. Then, there  exists a solution $U$ verifying $U_1\leq U \leq U_2$ in $\mathcal{C}_{\Omega}$.
\end{lemma}

Next we deal with a  comparison result.

\begin{lemma}\label{orden}
Let $U_1,U_2\in H_{\Sigma_{\mathcal{D}}^*}^1(\mathcal{C}_{\Omega},y^{1-2s}dxdy)$ be respectively a positive subsolution and a positive supersolution to $(P_f^*)$ and assume that $f(t)/t$ is decreasing for $t>0$. Then $U_1\leq U_2$ in $\mathcal{C}_{\Omega}$.
\end{lemma}
\begin{proof}
The proof is similar to the proof of \cite[Lemma 3.3]{ABC}. By definition we have, for any  positive test functions $\phi_1,\:\phi_2\in H_{\Sigma_{\mathcal{D}}^*}^1 (\mathcal{C}_{\Omega})$ that
\begin{equation*}
\begin{split}
\kappa_s\int_{\mathcal{C}_{\Omega}}y^{1-2s}\nabla U_1\nabla \phi_1 dxdy&\leq\int_{\Omega}f(u_1)\phi_1 (x,0)dx,\\
\kappa_s\int_{\mathcal{C}_{\Omega}}y^{1-2s}\nabla U_2\nabla \phi_2 dxdy&\geq\int_{\Omega}f(u_2)\phi_2 (x,0) dx,
\end{split}
\end{equation*}
where $u_1=U_1(x,0)$ and $u_2=U_2(x,0)$. Let $\theta(t)$ be a smooth non-decreasing function such that $\theta(t)=0$ for $t\leq0$, $\theta(t)=1$ for $t\geq1$, set $\theta_{\varepsilon}(t)=\theta(t/\varepsilon)$, and define the test functions $\varphi_1$ and $\varphi_2$ as
\begin{equation*}
\varphi_1=U_2\theta_{\varepsilon}\left(U_1-U_2\right),\quad \varphi_2=U_1\theta_{\varepsilon}\left(U_1-U_2\right).
\end{equation*}
From the above inequalities we obtain
\begin{equation*}
\begin{split}
J_{\varepsilon}:&=\kappa_s\int_{\mathcal{C}_{\Omega}}y^{1-2s}\left\langle U_1\nabla U_2-U_2\nabla U_1,\nabla\!\left(U_1-U_2\right)\right\rangle\theta_{\varepsilon}'\left(U_1-U_2\right)dxdy\\
&\geq\int_{\Omega}u_1u_2\left(\frac{f(u_2)}{u_2}-\frac{f(u_1)}{u_1}\right)\theta_{\varepsilon}\left(u_1-u_2\right)dx.
\end{split}
\end{equation*}
On the other hand,
\begin{equation*}
\begin{split}
J_{\varepsilon}&\leq \kappa_s\int_{\mathcal{C}_{\Omega}}y^{1-2s}\left\langle\nabla U_1,\left(U_1-U_2\right)\nabla\!\left(U_1-U_2\right)\right\rangle\theta_{\varepsilon}'\left(U_1-U_2\right)dxdy\\
&=\kappa_s\int_{\mathcal{C}_{\Omega}}y^{1-2s}\left\langle\nabla U_1,\nabla \eta_{\varepsilon}(U_1-U_2)\right\rangle dxdy\\
&=\int_{\Omega}f(u_1)\eta_{\varepsilon}(u_1-u_2)dx,
\end{split}
\end{equation*}
where $\eta_{\varepsilon}'(t)=t\theta_{\varepsilon}'(t)$. Since $0\leq\eta_{\varepsilon}\leq\varepsilon$, we find $I_{\varepsilon}\leq c\varepsilon$. Then, letting $\varepsilon\to0^+$ we conclude
\begin{equation*}
\int\limits_{\Omega\cap\{u_1>u_2\}}u_1u_2\left(\frac{f(u_2)}{u_2}-\frac{f(u_1)}{u_1}\right)dx\leq0.
\end{equation*}
Taking in mind the hypotheses on $f$, it follows $u_1\leq u_2$ in $\Omega$. The result for the whole cylinder $\mathcal{C}_{\Omega}$ follows by the maximum principle.
\end{proof}
Next we focus on the remaining assertions in Theorem \ref{Thuno}-$(2)$. Thus, from now on we assume that $0<q<1$.
\begin{lemma} \label{lem:Lambda}
Let $\Lambda$ be defined by
\begin{equation*}
\Lambda=\sup\{\lambda>0:(P_{\lambda})\  \text{has solution} \},
\end{equation*}
then, $0<\Lambda<\infty$.
\end{lemma}
\begin{proof}
As for the linear case, consider the eigenvalue problem associated
to the first eigenvalue $\lambda_1^s$, and let $\varphi_1$ the
associated eigenfunction. Using $\varphi_1$ as a test function in
problem \eqref{problema_abajo}, we have
\begin{equation}\label{truquillo}
\int_{\Omega}(\lambda
u^q+u^r)\varphi_1dx=\lambda_1^s\int_{\Omega}u\varphi_1dx.
\end{equation}
Since there exists a constant   $c=c(r,q)>1$ such that $\lambda
t^q+t^r>c\lambda^{\delta}t$ with $\delta=\frac{r}{r-q}$, for any
$t>0$, from \eqref{truquillo} we deduce
$c\lambda^{\delta}<\lambda_1^s$ and hence $\Lambda<\infty$. In
particular, this also proves that there is no solution to
\eqref{problema_abajo} for $\lambda>\Lambda$.

In order to prove  that $\Lambda>0$, we prove, by means of the sub  and supersolution technique, the existence of solution to \eqref{eq:funcional}
for any small positive $\lambda$. Indeed, for $\varepsilon>0$ small enough, $\underline{U}=\varepsilon E_s[\varphi_1]$ is a subsolution to
\eqref{eq:funcional}. A supersolution can be constructed as an appropiate multiple of the function $G$, the  solution to
\begin{equation*}
        \left\{
        \begin{array}{rlcl}
        \displaystyle -\text{div}(y^{1-2s}\nabla G)&\!\!\!\!=0  & & \mbox{ in } \mathcal{C}_{\Omega} , \\
        \displaystyle B(G)&\!\!\!\!=0   & & \mbox{ on } \partial_L\mathcal{C}_{\Omega} , \\
         \displaystyle \frac{\partial G}{\partial \nu^s}&\!\!\!\!= 1& &  \mbox{ on } \Omega\times\{y=0\}.
         \end{array}
         \right.
\end{equation*}
Since the trace function $g(x)=G(x,0)$ is a solution to
\begin{equation*}
        \left\{
        \begin{tabular}{rcl}
        $(-\Delta)^sg=1$ & &in $\Omega$, \\
        $B(g)=0$  & &on $\partial\Omega$,
        \end{tabular}
        \right.
\end{equation*}
because of \cite[Theorem 3.4]{CCLO} we have $\|g\|_{L^{\infty}(\Omega)}<+\infty$. Next, since $0<q<1<r$ we can find $\lambda_0>0$ such that for all $0<\lambda\leq\lambda_0$ there exists $M=M(\lambda)$ such that
\begin{equation}\label{eq:Mlto0}
M\geq\lambda M^q\|g\|_{L^{\infty}(\Omega)}^q+M^r\|g\|_{L^{\infty}(\Omega)}^r.
\end{equation}
As a consequence, the function $h=Mg$ satisfies $M=(-\Delta)^sh\geq \lambda h^q+h^r$ and, by the maximum principle, the extension function
$\overline U=E_s[h]$ is a supersolution and $\underline U \leq \overline U$. Applying Lemma \ref{existencia} we conclude the existence of a
solution $U$ to problem \eqref{eq:funcional}. Therefore, its trace $u(x)=U(x,0)$ is a solution to problem \eqref{problema_abajo}, $\lambda<\lambda_0$.
\end{proof}

\begin{remark}
Although Lemma \ref{lem:Lambda} provides  the existence of a solution for  small $\lambda>0$,  we can also prove this result studying the
associated functional $I_\lambda$. Indeed,
\begin{equation*}
\begin{aligned}
 I_{\lambda}(v)&= \frac 12\int_{\Omega}|(-\Delta)^{s/2}v|^2dx-\frac{\lambda}{q+1}\int_{\Omega} |v|^{q+1}dx-\frac{1}{r+1}\int_{\Omega}|v|^{r+1} dx \\
               &\geq \frac 12\int_{\Omega}|(-\Delta)^{s/2}v|^2dx - \lambda c_1 {\left( \int_{\Omega}|(-\Delta)^{s/2}v|^2dx \right)}^{(q+1)/2} \\
 &\mkern+25mu - c_2{\left( \int_{\Omega}|(-\Delta)^{s/2}v|^2dx \right)}^{(r+1)/2},
\end{aligned}
\end{equation*}
for some positive constants $c_1$ and $c_2$. Then, for sufficiently small $\lambda$, there exist (at least) two solutions to problem
\eqref{problema_abajo}, one given by minimization and another given by the Mountain-Pass Theorem. The proof is rather common, based on
the geometry of the function $g(t)= \frac{1}{2}t^2 - \lambda c_1t^{q+1} - c_2t^{r+1}$ (see for instance \cite{AR}).
\end{remark}

Next we show that there exists a solution for every $\lambda \in (0,\Lambda)$.
\begin{lemma}\label{lem:minimal}
Problem $(P_{\lambda})$ has at least a positive minimal solution for every $0<\lambda<\Lambda$. Moreover, the family $\{u_{\lambda}\}$ of minimal
solutions is increasing with respect to $\lambda$.
\end{lemma}
\begin{proof}
By definition of $\Lambda$, for any $0<\lambda<\Lambda$ there exists $\mu\in (\lambda,\Lambda]$ such that $(P_\mu^*)$ admits a solution $U_{\mu}$.
It is easy to see that  $U_{\mu}$ is a supersolution for   \eqref{eq:funcional}. On the other hand, let $V_{\lambda}$ be the unique solution to problem $(P_f^*)$ with $f(t)= \lambda t^q$ (the existence can be deduced by minimization, while uniqueness follows from Lemma~\ref{orden}). It is clear that $V_{\lambda}$ is a subsolution to problem \eqref{eq:funcional} and, because of Lemma \ref{orden}, we have $V_{\lambda} \leq U_{\mu}$. Therefore, by Lemma \ref{existencia}, we conclude that there is a solution to \eqref{eq:funcional} and, as a consequence, for the whole open interval $(0,\Lambda)$. Finally, we prove the existence of a minimal solution for all $0<\lambda<\Lambda$. Indeed, given a solution $u$ to \eqref{problema_abajo} we take $U=E_s(u)$ and, by Lemma~\ref{orden} being $U$ solution to problem \eqref{eq:funcional}, it satisfies $V_{\lambda}\leq U$ with $V_{\lambda}$ solution to problem $(P_f^*)$ with $f(t)= \lambda t^q$. Then, the function $v_{\lambda}(x)=V_{\lambda}(x,0)$ is a subsolution of problem \eqref{problema_abajo}
and the monotone iteration procedure described by
\begin{equation*}
\begin{array}{c}
(-\Delta)^s u_{n+1}=\lambda u_{n}^q+u_n^r,\quad  u_n \in H_{\Sigma_{\mathcal{D}}}^s ({\Omega})\quad \mbox{ with } \quad u_0=v_{\lambda},
 \end{array}
\end{equation*}
verifies $u_n\leq U(x,0)=u$ and $u_n\nearrow u_{\lambda}$ with $u_{\lambda}$ solution to problem \eqref{problema_abajo}. In particular $u_\lambda\leq u$ and we conclude that $u_\lambda$ is a minimal solution. The monotonicity follows directly from first part of the proof, taking $U_\mu=E_s(u_\mu)$ which leads to $u_\lambda\leq u_\mu$ whenever $0<\lambda < \mu\leq \Lambda$.
\end{proof}
\begin{remark}
In the proof of Lemma \ref{lem:Lambda}, precisely in \eqref{eq:Mlto0}, we can choose $M=M(\lambda)$  verifying $M(\lambda)\to 0$ as
$\lambda\to0$, proving that $\|u_{\lambda}\|_{L^{\infty}(\Omega)}\to0$ as $\lambda\to0$. Indeed, it is enough to choose $M(\lambda)=\lambda^\eta$ with $0<\eta<\frac{1}{1-q}$.
\end{remark}
\begin{lemma}\label{lem:solLambda}
Problem \eqref{eq:funcional} has at least one solution if $\lambda = \Lambda$.
\end{lemma}
To prove Lemma \ref{lem:solLambda} we extend \cite[Lemma 3.5]{ABC} to the fractional framework in this manuscript. This result guarantees that the linearized equation corresponding to \eqref{problema_abajo} has non-negative eigenvalues at the minimal solution.
\begin{prop}
Let $u_{\lambda}$ be the minimal solution to \eqref{problema_abajo} and define $a_{\lambda}=a_{\lambda}(x)=\lambda q u_{\lambda}^{q-1}+r u_{\lambda}^{r-1}$. Then, the operator $[(-\Delta)^s-a_\lambda(x)]$ with mixed boundary conditions has a first eigenvalue $\nu_1\geq0$.
\end{prop}
\begin{remark}
In particuar it follows that
\begin{equation}\label{ineq:posit}
\int_{\Omega}\left(|(-\Delta)^{s/2}v|^2-a_{\lambda}v^2\right)dx\geq0,\quad\mbox{for all } v\in H_{\Sigma_{\mathcal{D}}}^s(\Omega).
\end{equation}

\end{remark}
\begin{proof}
By contradiction, assume that $\nu_1<0$ and let $\phi_1>0$ be the first eigenfunction. Let $\alpha>0$ and observe that since $0<q<1$,
\begin{equation*}
\begin{split}
(-\Delta)^s(u_{\lambda}&-\alpha\phi_1)-\left(\lambda(u_{\lambda}-\alpha\phi_1)^q+(u_{\lambda}-\alpha\phi_1)^r\right)\\
&=\lambda u_{\lambda}^q+u_{\lambda}^r-\alpha\nu_1\phi_1-\alpha\left(\lambda q u_{\lambda}^{q-1}+r u_{\lambda}^{r-1}\right)\phi_1-\lambda(u_{\lambda}-\alpha\phi_1)^q-(u_{\lambda}-\alpha\phi_1)^r\\
&\geq u_{\lambda}^r-\alpha\nu_1\phi_1-\alpha r u_{\lambda}^{r-1}\phi_1-(u_{\lambda}-\alpha\phi_1)^r\\
&=-\alpha\nu_1\phi_1+o(\alpha\phi_1).
\end{split}
\end{equation*}
Using that $\nu_1<0$, $\phi_1>0$, for $\alpha>0$ sufficiently small we have that
$$
(-\Delta)^s(u_{\lambda}-\alpha\phi_1)-\left(\lambda(u_{\lambda}-\alpha\phi_1)^q+(u_{\lambda}-\alpha\phi_1)^r\right)\ge 0,
$$
proving that $u_{\lambda}-\alpha\phi_1$ is a supersolution.

Now, let $\psi=\lambda^{\frac{1}{q-1}}v$, with $v$ a solution to
\begin{equation}\label{prob:sublinear}
        \left\{
        \begin{tabular}{rcl}
        $(-\Delta)^sv= v^q$ & &in $\Omega$, \\
        $B(v)=0\mkern+7.6mu $           & &on $\partial\Omega$.
        \end{tabular}
        \right.
\end{equation}
 Then $\psi\leq u_{\lambda}-\alpha\phi_1$ and problem \eqref{problema_abajo} has a solution $\widetilde{u}$ such that $\psi\leq \widetilde{u}\leq u_{\lambda}-\alpha\phi_1$ in contradiction with the minimality of $u_{\lambda}$.
\end{proof}

\begin{proof}[Proof of Lemma \ref{lem:solLambda}]
Let $\{\lambda_n\}$ be a sequence such that $\lambda_n \nearrow \Lambda$ and denote by $u_n=u_{\lambda_n}$ the minimal solution to problem $(P_{\lambda_n})$.
Let $U_n=E_s[u_n]$, then
\begin{equation*}
I_{\lambda_n}(u_n)=\frac{1}{2}\int_{{\Omega}}|(-\Delta)^{\frac{s}{2}}u_n|^2dx-\frac{\lambda_n}{q+1}\int_{\Omega} u_n^{q+1}dx-\frac{1}{r+1}\int_{\Omega}u_n^{r+1}dx.
\end{equation*}
Moreover, as $u_n$ is a solution to \eqref{problema_abajo}, it also satisfies
\begin{equation*}
\int_{{\Omega}}|(-\Delta)^{\frac{s}{2}}u_n|^2dx=\lambda_n\int_{\Omega} u_n^{q+1}dx+\int_{\Omega}u_n^{r+1}dx.
\end{equation*}
On the other hand, using \eqref{ineq:posit} with $v=u_n$,
\begin{equation*}
\int_{{\Omega}}|(-\Delta)^{\frac{s}{2}}u_n|^2dx-\lambda_nq\int_{\Omega} u_n^{q+1}dx-r\int_{\Omega}u_n^{r+1}dx\geq0.
\end{equation*}

As in \cite[Lemma 3.5]{ABC}, we  conclude $I_{\lambda_n}(u_n)<0$. Since $I_{\lambda_n}'(u_n) = 0$, plainly we obtain that
${\|u_n\|}_{H_{\Sigma_{\mathcal{D}}}^s(\Omega)}\leq C$. Hence, there exists a weakly convergent subsequence $u_n\to u\in
H_{\Sigma_{\mathcal{D}}}^s(\Omega)$ and, as a consequence, $u$ is a weak solution of \eqref{problema_abajo} for $\lambda=\Lambda$.
\end{proof}

Next we assure the existence of a second solution to \eqref{problema_abajo} for every $0<\lambda<\Lambda$ following the ideas of \cite{AL}, developed to concave-convex problems in \cite{ACP,BCdPS} for the classical Laplacian and the fractional Laplacian respectively. In order to find a second solution by means of variational methods it is essential to have a first solution which is also a local minimum of the associated functional $J_{\lambda}$.

\begin{lemma}\label{lem:segsol}
Problem $(P_{\lambda})$ has at least two solutions for each $\lambda \in (0,\Lambda)$.
\end{lemma}

\begin{proof}
The proof follows exactly as in \cite{BCdPS}, Lemma 5.11.
\end{proof}
\subsection{Moving the boundary conditions}

\

Now we prove Theorem \ref{teo:movbc}, i.e., the assertions on the behavior of the minimal and mountain pass solutions when we move the boundary conditions (see hypotheses $(B_1)$-$(B_3)$). To this aim, we need the following result.
\begin{lemma}
Let $v$ be the solution to problem \eqref{prob:sublinear}. There exists a constant $\beta>0$ such that
\begin{equation}\label{ineq:forbound}
\|\phi\|_{H_{\Sigma_{\mathcal{D}}}^s(\Omega)}^2-q\int_{\Omega}v^{q-1}\phi^2dx\geq\beta\|\phi\|_{L^2(\Omega)}^2,\quad\mbox{for all } \phi\in H_{\Sigma_{\mathcal{D}}}^s(\Omega).
\end{equation}
\end{lemma}
\begin{proof}
Since we always consider boundary conditions such that $|\Sigma_{\mathcal{D}}|=\alpha>0$, the function $v$ can be obtained as
\begin{equation*}
\min\left\{\|\phi\|_{H_{\Sigma_{\mathcal{D}}}^s(\Omega)}^2-\frac{1}{q+1}\|\phi\|_{L^{q+1}(\Omega)}^{q+1}:\phi\in H_{\Sigma_{\mathcal{D}}}^s(\Omega) \right\},
\end{equation*}
and thus,
\begin{equation*}
\|\phi\|_{H_{\Sigma_{\mathcal{D}}}^s(\Omega)}^2-q\int_{\Omega}v^{q-1}\phi^2dx\geq0,\quad\mbox{for all } \phi\in H_{\Sigma_{\mathcal{D}}}^s(\Omega).
\end{equation*}
As a consequence, the linearized problem
\begin{equation} \label{prob:qlinealiz}
        \left\{
        \begin{tabular}{rcl}
        $(-\Delta)^s\varphi-qv^{q-1}\varphi=\mu\varphi$ & &in $\Omega$, \\
        $B(\varphi)=0\mkern+8mu $           & &on $\partial\Omega$,
        \end{tabular}
        \right.
\end{equation}
has a non-negative first eigenvalue $\mu_1$. Let $\varphi_1$ be the first eigenfunction and assume $\mu_1=0$. Since $v$ is a solution to \eqref{prob:sublinear}, then
\begin{equation*}
q\int_{\Omega}v^q\varphi_1 dx=\int_{\Omega}v^q\varphi_1 dx.
\end{equation*}
which is a contradiction. Hence $\mu_1>0$.
\end{proof}

\begin{lemma}\label{lem:bounduniq}
There exists $A>0$ such that for all $\lambda\in(0,\Lambda)$ the problem \eqref{problema_abajo} has at most one solution satisfying $\|u\|_{L^{\infty}(\Omega)}<A$.
\end{lemma}
\begin{proof}
Let $A>0$ such that $rA^{r-1}<\beta$, with $\beta$ given by \eqref{ineq:forbound}. Assumme by contradiction that there exists a second solution $u=u_{\lambda}+w$ of \eqref{problema_abajo} such that $\|u\|_{L^{\infty}(\Omega)}\leq A$. Since $u_{\lambda}$ is the minimal solution, $w\geq0$. Let $\zeta(x)=\lambda^{\frac{1}{1-q}}v(x)$ with $v$ the solution to \eqref{prob:sublinear}, so that $(-\Delta)^s\zeta=\lambda \zeta^q$. Moreover, $u_{\lambda}$ is also a supersolution of \eqref{prob:sublinear}, and hence, by Lemma \ref{orden}, $u_{\lambda}\geq \lambda^{\frac{1}{1-q}}v$. On the other hand, since $u=u_{\lambda}+w$ is a solution to \eqref{problema_abajo} we have
\begin{equation*}
(-\Delta)^s(u_{\lambda}+w)=\lambda(u_{\lambda}+w)^q+(u_{\lambda}+w)^r.
\end{equation*}
By concavity, $\lambda(u_{\lambda}+w)^q\leq\lambda u_{\lambda}^q+\lambda qu_{\lambda}^{q-1}w$ and hence
\begin{equation*}
(-\Delta)^sw\leq \lambda qu_{\lambda}^{q-1}w+(u_{\lambda}+w)^r-u_{\lambda}^r.
\end{equation*}
Furthermore, since $u_{\lambda}\geq \lambda^{\frac{1}{1-q}}v$, one also has $u_{\lambda}^{q-1}\leq \lambda^{-1}v^{q-1}$ and as we are assuming $\|u_\lambda\|_{L^{\infty}(\Omega)}\leq A$, we find
\begin{equation*}
\begin{split}
(-\Delta)^sw&\leq qv^{q-1}+(u_{\lambda}+w)^r-u_{\lambda}^r\\
&\leq qv^{q-1}+rA^{r-1}w.
\end{split}
\end{equation*}
Multiplying the above inequality by $w$ and using \eqref{ineq:forbound} we conclude
\begin{equation*}
\beta\int_{\Omega}w^2dx\leq rA^{r-1}\int_{\Omega}w^2dx.
\end{equation*}
Since $\beta<rA^{r-1}$, it follows $w=0$.
\end{proof}

Now we can perform the proof of Theorem~\ref{teo:movbc}.

\begin{proof}[Proof of Theorem~\ref{teo:movbc}]
First we claim that if $A=A(\alpha)$ is the associated constant to \eqref{pal} obtained in Lemma \ref{lem:bounduniq}, then $A(\alpha)\to0$ as $\alpha\to0$.

Indeed, it is enough to observe that
\begin{equation*}
0<\mu_1\leq\lambda_1^s(\alpha)=\inf_{\substack{u\in
H_{\Sigma_{\mathcal{D}}}^s(\Omega)\\ u\not \equiv
0}}\frac{\|u\|_{H_{\Sigma_{\mathcal{D}}}^s(\Omega)}^2}{\|u\|_{L^{2}(\Omega)}^2},
\end{equation*}
where $\mu_1$ is the first eigenvalue of the linearized eigenvalue problem \eqref{prob:qlinealiz}.

Since by Remark \ref{rem:sobconst} $\lambda_1^s(\alpha)$ as
$\alpha\searrow 0$, the result follows.

In particular we deduce:
\begin{enumerate}
\item From the proof of Lemma \ref{lem:Lambda}, we have
$c\Lambda^{\delta}(\alpha)<\lambda_1^s(\alpha)$ and arguing as
above $\Lambda(\alpha)\to0$ as $\alpha\to0$. \item There exist at
most one solution $u$ to \eqref{problema_abajo} with
$(\lambda,\|u\|_{\infty})\in(0,\Lambda(\alpha))\times(0,A(\alpha))$,
that is the minimal solution and, since $A(\alpha)\searrow 0$ as
$\alpha\to0$, the minimal solution converges to zero as $\alpha\searrow 0$.
\end{enumerate}

Now we prove that for $0<\lambda<\Lambda(\alpha)$ small enough, the solution to problem \eqref{pal} obtained by the Mountain Pass Theorem,  $u_{\alpha}$, satisfies
\begin{equation*}
\|u_{\alpha}\|_{H^{s}(\Omega)}\to0,\quad\mbox{as }\alpha\searrow 0.
\end{equation*}
The proof follows the lines of \cite[Lemma 5.12]{ColP}. Let us consider the funcional at $\lambda=0$
\begin{equation*}
\begin{aligned}
I_{0}(u_{\alpha})&=\frac{1}{2}\int_{{\Omega}} |(-\Delta)^{\frac s2}u_\alpha|^{2}dx -\frac{1}{r+1}\int_{\Omega}u_{\alpha}^{r+1}dx\\
&=\frac{1}{2}\|u_{\alpha}\|_{H_{\Sigma_{\mathcal{D}}}^s}^2-\frac{1}{r+1}\|u_{\alpha}\|_{L^{r+1}(\Omega)}^{r+1}\\
&\geq \frac{1}{2}\|u_{\alpha}\|_{H_{\Sigma_{\mathcal{D}}}^s}^2
-\frac{1}{r+1}|\Omega|^{1-\frac{r+1}{2_s^*}}\left(1+\frac{1}{\lambda_1^s(\alpha)}\right)^{\frac{r+1}{2}}\|u_{\alpha}\|_{H_{\Sigma_{\mathcal{D}}}^s}^{r+1}.
\end{aligned}
\end{equation*}
Let us define $g(t)=\frac12 t^2-c_2(r,|\Omega|)\lambda_{1}^{-s\frac{r+1}{2}}t^{r+1}$.
It is easy to see that if $t_{\alpha}$ is such that
$g'(t_{\alpha})=0$ then $t_{\alpha}\leq
c(r,|\Omega|)\lambda_{1}^{s\mu}(\alpha)$ with
$\mu=\frac{r+1}{2(r-1)}$, so that $t_{\alpha}\to0$ as
$\alpha\searrow  0$. Hence, the Mountain Pass solution converges to zero
as $\alpha\searrow 0$.
\end{proof}
\begin{remark}As a conclusion of the above arguments:
\begin{enumerate}
\item Both solutions, the minimal solution $u_{\lambda}$ and the mountain pass solution $u_{mp}$, converge to zero as $\alpha\searrow 0$.
\item If we set $\alpha\in I_{\varepsilon}=[\varepsilon,|\partial\Omega|]$ with $\varepsilon>0$, under hypotheses $(\mathfrak{B}_{\alpha})$ and
$(B_1)$-$(B_3)$, there exist $M_{\varepsilon},\,\Lambda_{\varepsilon}$ such that the family
$\mathcal{S}_{\varepsilon}\subset[0,\Lambda_{\varepsilon}]\times[0,M_{\varepsilon}]$.
\item To finish, it is interesting to point out Theorem 8 by
Denzler in \cite{De1}, where the author proved that
$$
\sup_{0<\alpha<|\partial\Omega|}\{\lambda_1(\alpha):
\alpha=|\Sigma_{\mathcal{D}}|\}=\lambda_1(|\partial \Omega|),
$$
which in particular proves that there are configurations about the
distribution of the manifolds $\Sigma_{\mathcal{D}}$ and
$\Sigma_{\mathcal{N}}$ on $\partial\Omega$ such that \cite[Lemma
4.1]{ColP} does not apply and hence $\lambda_1^s(\alpha)\not\to 0$
as $\alpha \searrow 0$. But this is not our case under hypotheses
$(\mathfrak{B}_{\alpha})$ and $(B_1)$-$(B_3)$, in which
\cite[Lemma 4.1]{ColP} applies proving that
$\lambda_1^s(\alpha)\to 0$ as $\alpha \searrow 0$.
\end{enumerate}
\end{remark}

\end{document}